\documentclass[a4paper,10pt]{article}
\usepackage[english]{babel}
\input xy
\xyoption{all}
\usepackage{amssymb,amsmath,amsthm,color}

\newcommand{\abs}[1]{\left\vert#1\right\vert}
\newcommand{\R}{\mathbb{R}}
\newcommand{\im}{\mathrm{Im}\,}         
\newcommand{\lie}[1]{\mathfrak{#1}}     
\newcommand{\Z}{\mathbb{Z}}

\newcommand{\C}{\mathbb{C}}

\newcommand{\RP}{\mathbb{RP}}

\newcommand{\hook}{\lrcorner\,}
\newcommand{\LieG}[1]{\mathrm{#1}}      

\newcommand{\SU}{\mathrm{SU}}

\newcommand{\SO}{\mathrm{SO}}

\newcommand{\so}{\mathfrak{so}}
\newcommand{\su}{\mathfrak{su}}

\newcommand{\GL}{\mathrm{GL}}

\newcommand{\dfn}[1]{\emph{#1}}
\newcommand{\id}{\mathrm{Id}}   

\newcommand{\gl}{\lie{gl}}
\newcommand{\Span}[1]{\operatorname{Span}\left\{#1\right\}}

\DeclareMathOperator{\End}{End}
\DeclareMathOperator{\Sym}{Sym}
\DeclareMathOperator{\Hom}{Hom}
\DeclareMathOperator{\Aut}{Aut}
\DeclareMathOperator{\Iso}{Iso}
\DeclareMathOperator{\ad}{ad}
\DeclareMathOperator{\Ad}{Ad}

\DeclareMathOperator{\diag}{diag}
\DeclareMathOperator{\rk}{rk}

\theoremstyle{plain}
\newtheorem{proposition}{Proposition}
\newtheorem{theorem}[proposition]{Theorem}
\newtheorem{lemma}[proposition]{Lemma}
\newtheorem{corollary}[proposition]{Corollary}

\theoremstyle{definition}
\newtheorem{definition}[proposition]{Definition}
\newtheorem{example}{Example}
\theoremstyle{remark}
\newtheorem*{remark}{Remark}

\title{$\SU(3)$-holonomy metrics from nilpotent Lie groups}
\author{Diego Conti}

\begin{document}

 \maketitle

\begin{abstract}
One way of producing explicit Riemannian $6$-manifolds with holonomy $\SU(3)$ is by integrating a flow of $\SU(2)$-structures on a $5$-manifold, called the hypo evolution flow. In this paper we classify invariant hypo $\SU(2)$-structures on nilpotent $5$-dimensional Lie groups. 
We characterize the hypo evolution flow in terms of gauge transformations, and study the flow induced on the variety of frames on a Lie algebra taken up to automorphisms. We classify the orbits of this flow for all hypo nilpotent structures, obtaining several families of cohomogeneity one metrics with holonomy contained in $\SU(3)$. We prove that these metrics cannot be extended to a complete metric, unless they are flat.
\end{abstract}
This paper uses the language of hypo geometry to construct $6$-manifolds with holonomy contained in $\SU(3)$.

Consider such a six-manifold $(M,g)$, and fix an oriented hypersurface $N$ in $M$. Under suitable assumptions, the exponential map allows one to identify a neighbourhood of $N$ in $M$ with $N\times(a,b)$, and put the metric $g$ into the ``generalized cylinder form'' $g_t^2+dt^2$ (see \cite{ContiSalamon, BarGauduchonMoroianu}); here $t$ parametrizes the interval $(a,b)$, and  $g_t$ is a one-parameter family of Riemannian metrics on $N$, identified with $N\times\{t\}$. More precisely, the holonomy reduction determines an integrable $\SU(3)$-structure on $M$, inducing in turn a \dfn{hypo} $\SU(2)$-structure on each $N\times\{t\}$, or equivalently a one-parameter family of hypo $\SU(2)$-structures on $N$. Integrability implies that this one-parameter family satisfies the \dfn{hypo evolution equations} \cite{ContiSalamon}, and vice versa. From this point of view, the $\SU(3)$ holonomy metric is determined by a one-parameter family of hypo $\SU(2)$-structures satisfying these equations.

In the case that a group acts on $M$ with cohomogeneity one, a generic orbit is a hypersurface $N$ with a \emph{homogeneous} hypo structure. The evolution equations are then reduced to an ODE. We are interested in solving explicitly these ODE's in the special case where the homogeneous orbit is a nilpotent Lie group, or equivalently a nilmanifold. It is known by \cite{ContiSalamon} that $6$ out of the $9$ simply connected real nilpotent Lie groups of dimension $5$ admit an invariant hypo structure. In this paper we refine this result, and classify the space of invariant hypo structures on each of the six Lie groups; moreover, we solve the evolution equations in each case, in a sense explained later in this introduction.

Thus, we obtain several families of metrics with holonomy contained in $\SU(3)$ on products $G\times I$, with $G$ a nilpotent Lie group and $I$ a real interval. Such a metric is complete if and only if $I=\R$ and the evolution is trivial; otherwise, one might hope to obtain a complete metric by adding one or two special orbits. Using the classification of orbits, we prove  this cannot be done. In other words, we prove that every complete six-manifolds with an integrable $\SU(3)$-structure preserved by the cohomogeneity one action of a nilpotent five-dimensional Lie group is flat.

\medskip
Throughout this paper, an $\SU(2)$-structure  on a $5$-manifold is a reduction to $\SU(2)$ of the frame bundle, relative to the chain of inclusions \[\SU(2)\subset\SO(4)\subset\SO(5)\subset \GL(5,\R).\]
Such a structure can be characterized in terms of a $1$-form $\alpha$ and three $2$-forms $\omega_1$, $\omega_2$, $\omega_3$, or equivalently by the choice of an orientation and three differential forms $(\omega_1,\psi_2,\psi_3)$, related to the others by $\psi_k=\omega_k\wedge\alpha$. In this paper we give an intrinsic characterization of triples of differential forms $(\omega_1,\psi_2,\psi_3)$ that actually determine an $\SU(2)$-structure.

Using this language, we give an alternative description of the hypo evolution equations in terms of a Hamiltonian flow inside the product of cohomology classes
\[[\omega_1]\times[\psi_2]\times[\psi_3]\times\left[\frac12\omega_1\wedge\omega_1\right],\]
in analogy with Hitchin's results concerning half-flat evolution \cite{Hitchin:StableForms}. Similar arguments lead us to restate the hypo evolution equations in terms of gauge transformations, in the same vein as \cite{Stock}.

In the special context of invariant structures on Lie groups --- which we view as structures on the associated Lie algebra $\lie{g}$ --- the language of gauge transformations makes hypo evolution into a flow on the space of frames $\Iso(\R^5,\lie{g})$. This flow is invariant  under both the structure group $\SU(2)$, acting on the right, and the group of automorphisms $\Aut(\lie{g})$, acting on the left. In fact the quotient $\Iso(\R^5,\lie{g})/\Aut(\lie{g})$ can be immersed naturally in the affine variety
\[\mathcal{D}=\{d\in\Hom(\R^5,\Lambda^2\R^5),d_2\circ d=0\},\]
where $d_2\colon\Lambda^2\R^5\to\Lambda^3\R^5$ is the linear map induced by $d$ via the Leibniz rule; one can think of $\mathcal{D}$ as the space of five-dimensional Lie algebras with a fixed frame taken up to automorphisms. The variety $\mathcal{D}$ is the natural setting to study geometric structures on Lie algebras; see for example \cite{Lauret}, or \cite{DeAndres:HypoContact}, where hypo-contact structures on solvable Lie groups are classified in terms of a frame adapted to the structure. It is therefore inside $\mathcal{D}$ that we carry out the classification of hypo structures and integral lines of the evolution flow.

In fact, we are only studying nilpotent Lie algebras, which correspond to a subvariety of $\mathcal{D}$. Hypo nilpotent Lie algebras give rise to three families in $\mathcal{D}$, closed under $\LieG{U}(2)$-action. These families can be identified by selecting a slice with respect to this action: the resulting sets we obtain are semi-algebraic in $\mathcal{D}$. What is more surprising is that the integral lines of the evolution flow turn out to be semi-algebraic sets themselves. The classification of hypo structures is based on both the methods of \cite{Salamon:ComplexStructures,ContiSalamon} and the more recent methods of \cite{ContiFernandezSantisteban}. The classification of orbits is a long but standard computation, relying on the determination of first integrals.

Computing the integral lines in $\mathcal{D}$ is not quite the same as solving the hypo evolution equations on each Lie algebra, but it is sufficient to determine many interesting properties of the resulting $\SU(3)$-holonomy metrics, like their curvature. This is explained in Section~\ref{sec:examples}, where we illustrate with examples how one can use the classification of the integral lines to determine whether the corresponding $6$-dimensional metric is irreducible.

In the final section of this paper, we prove that the cohomogeneity one metrics we have obtained cannot be extended by adding a special orbit. In particular, this means that they cannot be extended to a complete metric, unless they are complete to begin with, in which case they are flat.

\medskip
\noindent\textbf{Acknowledgements.}
The first part of this paper is based on the author's  \emph{Tesi di Perfezionamento} at Scuola Normale Superiore di Pisa \cite{thesis}, written under supervision of Simon Salamon.

\section{$\SU(2)$-structures revisited}
\label{sec:SU2}
We are interested in $\SU(2)$-structures on $5$-manifolds, where $\SU(2)$ acts on $\R^5$ as $\C^2\oplus\R$. In this section, we shall work at a point, that is to say on a real vector space $T=\R^5$. Thus, an $\SU(2)$-structure on $T$ consists in the choice of an \dfn{adapted} frame $u\colon \R^5\to T$, determined uniquely up to right $\SU(2)$ action.

Another convenient description introduced in \cite{ContiSalamon} arises from viewing $\SU(2)$ as the intersection of the stabilizers in $\GL(5,\R)$ of four specific elements of $\Lambda(\R^5)^*$, namely
\[e^5,\quad e^{12}+e^{34}, \quad e^{13}+e^{42}, \quad e^{14}+e^{23}.\]
Here $e^1,\dotsc, e^5$ denotes the standard basis of $(\R^5)^*$ and, say, $e^{12}$ stands for $e^1\wedge e^2$. Thus, an $\SU(2)$-structure on $T$ can be identified with a quadruple $(\alpha,\omega_1,\omega_2,\omega_3)$ of elements in $\Lambda T^*$, that are mapped respectively to the above elements, by the transpose $u^t\colon T^*\to(\R^5)^*$ of an adapted frame. Of course the forms $(\alpha,\omega_i)$ must satisfy certain conditions (see \cite{ContiSalamon}) for such a frame to exist. In this case, we say that $(\alpha,\omega_i)$ \dfn{defines} an $\SU(2)$-structure.

On the other hand, $\SU(2)$ is also the intersection of the stabilizers in $\GL^+(5,\R)$ of
\[e^{12}+e^{34}, e^{135}+e^{425}, e^{145}+e^{235}.\]
Thus, an $\SU(2)$-structure can be given alternatively by the choice of an orientation and a triple $(\omega_1,\psi_2,\psi_3)$ in $\Lambda T^*$.

It is straightforward to determine $(\omega_1,\psi_2,\psi_3)$ given either an adapted frame or the quadruplet $(\alpha,\omega_1,\omega_2,\omega_3)$, the correspondence being given by
\[\psi_2=\omega_2\wedge\alpha,\quad \psi_3=\omega_3\wedge\alpha.\]
In this section we determine intrinsic conditions for a triple $(\omega_1,\psi_2,\psi_3)$ to determine an $\SU(2)$-structure, and explicit formulae for recovering $\alpha$, $\omega_2$, $\omega_3$.

\smallskip
As a first step, we construct maps
\begin{align*}
\Lambda^2T^*&\to T\otimes\Lambda^5T^* & \Lambda^3T^*&\to T^*\otimes\Lambda^5T^*& \Lambda^2T^*\times\Lambda^3T^*&\to\Lambda^5T^*\\
\omega &\to X_\omega  & \psi&\to \alpha_\psi & (\omega,\psi)&\to V^2(\omega,\psi)
\end{align*}
To this end, fix an orientation on $T$ and consider the isomorphisms
\begin{gather*}
A:\Lambda^kT^*\to\Lambda^{5-k}T\otimes\Lambda^5T^*,\quad A^*:\Lambda^kT\to\Lambda^{5-k}T^*\otimes\Lambda^5T,\\
k!\langle\gamma, \eta\rangle = \eta\wedge A\gamma=\gamma\wedge A^*\eta, \quad \gamma\in\Lambda^kT^*, \eta\in\Lambda^{k}T.
\end{gather*}
In the definition of $A$ (and similarly for $A^*$) we are thinking of elements of $\Lambda^{5-k}T\otimes\Lambda^5T^*$ as multivectors whose length depends on the choice of a volume form. The pairing $\langle,\rangle$ is the usual pairing between $\Lambda^kT^*$ and $\Lambda^kT$, i.e.
\[\langle \eta_1\wedge \dotsb \wedge \eta_k, v_1\wedge \dotsb \wedge v_k\rangle=\frac1{k!}\det(\eta_i(v_j))_{ij}.\]
Moreover, we are silently performing the identification
\[\Lambda^5T^*\otimes\Lambda^5T\cong \R, \quad (\eta,v)\to 5!\langle \eta,v\rangle.\]
The coefficients $k!$ guarantee that if we fix a metric on $T$ and  identify $T=T^*$, $\Lambda^5T^*=\R$ accordingly, then $A$ is just the Hodge star. Indeed, using the fact that $T$ is odd-dimensional, we have the usual properties:
\begin{equation}
 \label{eqn:theA}
\begin{aligned}
k!\langle\gamma, A\beta\rangle &= \beta\wedge \gamma, & \gamma&\in\Lambda^kT^*, \beta\in\Lambda^{5-k}T^*;\\
(5-k)!\langle A\gamma, \beta\rangle& = k!\langle \gamma, A\beta\rangle, & \gamma&\in\Lambda^kT^*, \beta\in\Lambda^{5-k}T^*;\\
Y\wedge(A\psi)&=A(Y\hook\psi), &\psi&\in\Lambda^kT^*, Y\in T.
\end{aligned}
\end{equation}
Similar properties hold for $A^*$.

We can now set
\[X_\omega=A(\omega^2), \quad \alpha_\psi=A^*((A\psi)^2), \quad V^2(\omega,\psi)=\alpha_\psi(X_\omega).\]
Notice that $\omega$ and $\psi$ play a symmetric r\^ole, as one can make explicit by interchanging $T$ with $T^*$ and composing with $A$. If $V^2(\omega,\psi)$ admits a ``square root'' in $\Lambda^5T^*$,  we write $V^2(\omega,\psi)>0$; having fixed an orientation on $T$,  we define
$V(\omega,\psi)$ as the positively oriented square root of $V^2(\omega,\psi)$.
The condition \[\omega^2(A\psi\wedge A\psi)=V^2(\omega,\psi)>0\]
implies that both $\omega$ and $\psi$ are stable in the sense of Hitchin (\cite{Hitchin:StableForms}), i.e. their $\GL(5,\R)$-orbits are open. It also implies that
\[\ker\alpha_\psi \oplus \Span{X_\omega}\]
is a direct sum. Finally, it contains an orientation condition, as shown by the example
\[V^2(e^{12}+e^{34}, e^{125}-e^{345})<0.\]

We can now state and prove a condition for $(\omega_1,\psi_2,\psi_3)$ to define an $\SU(2)$-structure.
\begin{proposition}
\label{prop:SL2CAndSU2} Given $\omega_1$ in $\Lambda^2T^*$ and $\psi_2,\psi_3$ in $\Lambda^3T^*$, then $(\omega_1,\psi_2,\psi_3)$ determines an $\SU(2)$-structure, in the sense that
\[\omega_1=e^{12}+e^{34}, \psi_2=e^{135}+e^{425}, \psi_3=e^{145}+e^{235}\]
for some coframe $e^1,\dotsc, e^5$, if and only if
\begin{equation}
\label{eqn:omegapsipsi}
 \alpha_{\psi_2}=\alpha_{\psi_3},\quad \omega_1\wedge\psi_2=0=\omega_1\wedge\psi_3,\quad (X_{\omega_1}\hook\psi_2)\wedge\psi_3=0, \quad V^2(\omega_1,\psi_2)>0,
\end{equation}
and $\omega_1(Y,Z)\geq0$
 whenever  $Y,Z$ in $T$ satisfy $Y\hook\psi_2=Z\hook\psi_3$.
\end{proposition}
\begin{proof}
The ``only if'' part is straightforward. Conversely, suppose \eqref{eqn:omegapsipsi} holds. In particular
\[V^2(\omega_1,\psi_2)=V^2(\omega_1,\psi_3)>0,\]
so we may define $\alpha\in T^*$ and $X\in T$ by
\[
\alpha=V(\omega_1,\psi_2)^{-1}\alpha_{\psi_2},\quad X=V(\omega_1,\psi_2)^{-1}X_{\omega_1}.
\]
Then
\[T=\Span{X}\oplus\ker\alpha, \quad \alpha(X)=1.\]
This splitting reduces the structure group to $\GL(4,\R)$; in order to further reduce to
$\SU(2)$, set
\[\omega_2=X\hook\psi_2,\quad \omega_3=X\hook\psi_3.\]
By construction $\omega$ is a non-degenerate form on $\ker\alpha$ and
$X\hook\omega_1^2=0$, so $X\hook\omega_1$ is zero as well. By duality,
\[\alpha\wedge\psi_2=0=\alpha\wedge\psi_3,\]
thus
\[\psi_2=\alpha\wedge\omega_2, \quad \psi_3=\alpha\wedge\omega_3.\]
It follows easily from \eqref{eqn:theA} that
 \begin{equation}
 \label{eqn:VFromOmega}
 V(\omega_1,\psi_2)=\langle \alpha,X_{\omega_1}\rangle = \langle  \alpha, A(\omega_1^2)\rangle = \omega_1^2\wedge\alpha.
 \end{equation}
Similarly
 \begin{multline}
 \label{eqn:VFromPsi}
 V(\omega_1,\psi_2)=\langle \alpha_{\psi_2},X\rangle = \langle  A^*(A\psi_2)^2, X\rangle =
4!\langle (A\psi_2)^2, A^*X\rangle
=X\wedge (A\psi_2)^2\\
= (X\wedge A\psi_2) \wedge A\psi_2 = 3!\langle X\wedge A\psi_2, \psi_2\rangle= \langle A(X\hook\psi_2), \psi_2\rangle = \omega_2\wedge\psi_2.
 \end{multline}
In the same way, $V(\omega_1,\psi_3)=\omega_3\wedge\psi_3$. Therefore, using \eqref{eqn:omegapsipsi} again, we find
\[\omega_i\wedge\omega_j=X\hook(\omega_i\wedge\psi_j)=\delta_{ij}(\omega_1)^2.\]
Thus there is a metric on $\ker\alpha$ such that $(\omega_1,\omega_2,\omega_3)$ is an orthonormal basis of $\Lambda^2_+(\ker\alpha)$. It is also positively oriented by the last condition in the hypothesis.
\end{proof}
\begin{remark}
 The last condition in the hypothesis of Proposition~\ref{prop:SL2CAndSU2} is an open condition.
\end{remark}

Since $X_\omega$ only depends on $\omega^2$,  it also makes sense to define $X_\upsilon$ for $\upsilon\in\Lambda^4T^*$. In order
to make subsequent formulae simpler, we introduce a coefficient and set
\[X_\upsilon=2A(\upsilon).\]
Given also $\psi$ in $\Lambda^3T^*$, we set
\[V^2(\psi,\upsilon)=\alpha_\psi(X_\upsilon),\]
and, as usual, if $V^2(\psi,\upsilon)>0$ we denote by $V(\psi,\upsilon)$ its positive square root.

\section{Hypo evolution revisited}
\label{sec:hamiltonian}
The definition of an $\SU(2)$-structure on a vector space can be extended immediately to a definition on a $5$-manifold: thus, we can identify an $\SU(2)$-structure on an oriented $5$-manifold with a triple $(\omega_1,\psi_2,\psi_3)$ of differential forms, satisfying at each point the conditions of Proposition~\ref{prop:SL2CAndSU2}. Equivalently, we can use the quadruple $(\alpha,\omega_1,\omega_2,\omega_3)$.

An $\SU(2)$-structure is called \dfn{hypo} if the forms $\omega_1$, $\psi_2$, $\psi_3$ are closed.
Hypo structures (see \cite{ContiSalamon}) arise naturally on hypersurfaces inside $6$ manifolds with holonomy contained in $\SU(3)$; a partial converse holds, in the sense that given a compact real analytic $5$-manifold $M$ and a real analytic hypo structure, there always exists a one-parameter family of hypo structures $(\alpha(t),\omega_i(t))$, coinciding with the original hypo structure at time zero, inducing an integrable $\SU(3)$-structure on the ``generalized cylinder'' $M\times(a,b)$; explicitly, the K\"ahler form and complex volume on $M\times(a,b)$ are given by
\[\omega_1(t)+\alpha(t)\wedge dt, \quad \Psi(t)=(\omega_2(t)+i\omega_3(t))\wedge(\alpha(t)+idt).\]
This one-parameter family satisfies the \dfn{hypo evolution equations}
\begin{equation}
\label{eqn:hypoevolution}
 \left\{
\begin{aligned}
 \frac{\partial}{\partial t}\omega_1&=-d\alpha\\
 \frac{\partial}{\partial t}(\omega_2\wedge\alpha)&=-d\omega_3\\
 \frac{\partial}{\partial t}(\omega_3\wedge\alpha)&=d\omega_2\\
\end{aligned}
 \right.
\end{equation}
It will be understood that a solution $(\alpha(t),\omega_i(t))$ of \eqref{eqn:hypoevolution} is required to define an $\SU(2)$-structure for all $t$.

In this section we will show that the solutions of this evolution equations, that we know exist, can be viewed as integral lines of a certain Hamiltonian vector field, in analogy with half-flat evolution (see \cite{Hitchin:StableForms}).

\smallskip
Fix a compact 5-dimensional manifold $M$; let $B^p$ be the space of exact $p$-forms on $M$. The evolution equations \eqref{eqn:hypoevolution} can be viewed as a flow of \[(\omega_1(t),\psi_2(t),\psi_3(t))\in \Omega^2(M)\times\Omega^3(M)\times\Omega^3(M);\] however, the cohomology class of $\omega_1(t)$, $\psi_2(t)$ and $\psi_3(t)$ is independent of $t$, so the flow stays inside the product of the three cohomology classes.

In order to obtain a symplectic structure, we need to add a fourth cohomology class. Let $(\tilde\omega_1,\tilde\psi_2,\tilde\psi_3)$ define a hypo $\SU(2)$-structure on  $M$, and set
\[\mathcal{\tilde H}=(\tilde\omega_1+B^2)\times (\tilde\psi_2+B^3)\times (\tilde\psi_3+B^3)\times \left(\frac{1}{2}\tilde\omega^2+B^4\right)\;.\]
This is a Fr\'echet space with the $C^\infty$ topology; it contains an open set
\[\mathcal{H}=\{(\omega_1,\psi_2,\psi_3,\nu)\in\mathcal{\tilde H}\mid V^2(\omega_1,\psi_2)>0, V^2(\psi_3,\upsilon)>0\}.\]
The skew-symmetric form on $B^2\times B^3\times B^3\times B^4$ defined by
\[\left\langle(\dot\omega,\dot\psi_2,\dot\psi_3,\dot\upsilon),(d\beta,d\tau_2,d\tau_3,d\gamma)\right\rangle=
    \int_M\left(\dot\omega\wedge\gamma-\dot\upsilon\wedge\beta-\dot\psi_2\wedge\tau_3-\dot\psi_3\wedge\tau_2\right)\]
makes $\mathcal{H}$ into a symplectic manifold.

We will say that a point $(\omega_1,\psi_2,\psi_3,\upsilon)$ of $\mathcal{H}$ defines an $\SU(2)$-structure if $\upsilon=\frac12\omega_1^2$ and the forms $\omega_1$, $\psi_2$, $\psi_3$ satisfy the conditions of Proposition~\ref{prop:SL2CAndSU2}.
\begin{remark}
Symplectic interpretation aside, the presence of the redundant differential form $\upsilon$ can be motivated by the fact that
an $\SU(3)$\nobreakdash-structure on a 6-manifold is determined by a 3-form and a 4-form, i.e. a
section of $\Lambda^3\oplus\Lambda^4$, and the pullback of this vector bundle to a hypersurface splits up as
$\Lambda^2\oplus\Lambda^3\oplus\Lambda^3\oplus\Lambda^4$.
\end{remark}

We can now state the main result of this section.
\begin{theorem}
\label{thm:hamiltonian}
Solutions of the hypo evolution equations~\eqref{eqn:hypoevolution} are integral lines of the Hamiltonian flow of the functional
\begin{equation}
 \label{eqn:functional}
H:\mathcal{H}\to\R, \quad H(\omega_1,\psi_2,\psi_3,\upsilon)=\int_M\left(V(\omega_1,\psi_2)-V(\psi_3,\upsilon)\right).
\end{equation}
\end{theorem}
Before proving the theorem we need a few lemmas. The first lemma makes use of the fact that $V^2(\omega,\psi_2)$ and $V^2(\psi_3,\upsilon)$ are positive on $\mathcal{H}$ in order to define additional differential forms.
\begin{lemma}
\label{lemma:apointofH}
Let  $(\omega_1,\psi_2,\psi_3,\upsilon)$ be a point of $\mathcal{H}$, and let
\begin{align*}
\alpha_2&=V^{-1}(\omega_1,\psi_2)\alpha_{\psi_2},&\alpha_3&=V^{-1}(\psi_3,\upsilon)\alpha_{\psi_3},\\
\omega_2&=\left(V^{-1}(\omega_1,\psi_2)X_{\omega_1}\right)\hook\psi_2,&
\omega_3&=\left(V^{-1}(\psi_3,\upsilon)X_{\upsilon}\right)\hook\psi_3.
\end{align*}
If $(\omega_1,\psi_2,\psi_3,\upsilon)$ defines an $\SU(2)$-structure, then $\alpha_2=\alpha_3$ and $(\alpha_2,\omega_1,\omega_2,\omega_3)$ defines the same structure.
\end{lemma}
\begin{proof}
By construction $X_{\omega_1}=X_{\upsilon}$, so $V(\omega_1,\psi_2)=V(\psi_3,\upsilon)$. Proposition~\ref{prop:SL2CAndSU2} concludes the proof.
\end{proof}

In the next lemma we work at a point, and compute the differential of
\[V\colon \Omega_1\to \Lambda^5T^*, \quad \Omega_1=\{(\omega,\psi)\in \Lambda^2T^*\times\Lambda^3T^*\mid V^2(\omega,\psi)>0\}.\]
The differential of this map at a point is an element of
\[(\Lambda^2T^*\oplus\Lambda^3T^*)^*\otimes\Lambda^5T^*\cong(\Lambda^2T\oplus\Lambda^3T)\otimes\Lambda^5T^*\cong\Lambda^3T^*\oplus\Lambda^2T^*.\]
Similarly, the differential at a point of
\[V\colon \Omega_2\to \Lambda^5T^*, \quad \Omega_2=\{(\psi,\upsilon)\in \Lambda^3T^*\times\Lambda^4T^*\mid V^2(\psi,\upsilon)>0\}\]
is an element of $\Lambda^2T^*\oplus T^*$.
\begin{lemma}
\label{lemma:dV} Given $\omega$, $\psi$ and $\upsilon$ such that $(\omega,\psi)$ is in $\Omega_1$ and $(\psi,\upsilon)$ is in $\Omega_2$, set
\[\hat\omega=\left(V^{-1}(\omega,\psi)\alpha_\psi\right)\wedge\omega, \quad \hat\psi=\left(V^{-1}(\omega,\psi)X_\omega\right)\hook\psi, \quad
\check\psi=\left(V^{-1}(\psi,\upsilon)X_{\upsilon}\right)\hook\psi;\]
then
\[
 dV_{(\omega,\psi)}(\sigma,\phi)=\hat\omega\wedge\sigma+\hat\psi\wedge\phi, \quad
dV_{(\psi,\upsilon)}(\phi,\sigma)=\check\psi\wedge\phi+\left(V^{-1}(\psi,\upsilon)\alpha_\psi\right)\wedge\sigma.
\]
\end{lemma}
\begin{proof}
Multiplying (\ref{eqn:VFromOmega}) by $V(\omega,\psi)$ and differentiating with respect to $\omega$, we find
\[2V(\omega,\psi)dV_{(\omega,\psi)}(\sigma,0)=2\,\omega\wedge\sigma\wedge\alpha_\psi,\] so
\[dV_{(\omega,\psi)}(\sigma,0)=\omega\wedge\sigma\wedge\left(V^{-1}(\omega,\psi)\alpha_\psi\right).\]
Similarly, \eqref{eqn:VFromPsi} gives $V^2(\omega,\psi)=(X_\omega\hook\psi)\wedge\psi$, and therefore
\[dV_{(\omega,\psi)}(0,\phi)=\left(V^{-1}(\omega,\psi)X_\omega\hook\psi\right)\wedge\phi.\]
Summing the two equations determines the first formula for the differential of $V$; the second is completely analogous.
\end{proof}
 \begin{lemma}
The skew gradient of the functional $H$ defined in \eqref{eqn:functional} is
\begin{equation}
\label{eqn:X_H} (X_H)_{(\omega_1,\psi_2,\psi_3,\upsilon)}=(-d\alpha_3,-d\omega_3,d\omega_2,-\omega_1\wedge d\alpha_2)\;.
\end{equation}
\end{lemma}
\begin{proof}
By Lemma \ref{lemma:dV},
\begin{multline*}dH_{(\omega_1,\psi_2,\psi_3,\upsilon)}(d\beta,d\tau_2,d\tau_3,d\gamma)\\
=\int_M(\alpha_2\wedge\omega_1\wedge d\beta+\omega_2\wedge
d\tau_2-\omega_3\wedge d\tau_3-\alpha_3\wedge d\gamma)\\
=\int_M(d\alpha_2\wedge\omega_1\wedge \beta-d\omega_2\wedge \tau_2+d\omega_3\wedge \tau_3-d\alpha_3\wedge
\gamma)\;\text{ by Stokes' theorem.}
\end{multline*}
The vector field $X_H$ given in \eqref{eqn:X_H} clearly satisfies $\langle X_H,\cdot\rangle=dH$.
\end{proof}

\begin{proof}[Proof of Theorem~\ref{thm:hamiltonian}]
Given a one-parameter family of $\SU(2)$-structures that evolves according to \eqref{eqn:hypoevolution}, we must verify that the corresponding curve in $\mathcal{H}$ is an integral line of the skew gradient  \eqref{eqn:X_H}. It is sufficient to prove it at a point.

Let $(\omega_1,\psi_2,\psi_3,\upsilon)$ be a point of $\mathcal{H}$ that defines an $\SU(2)$-structure. By Lemma~\ref{lemma:dV}, the curve is an integral line of the Hamiltonian flow if
\begin{align*}
 \frac{\partial}{\partial t}\omega_1&=-d\alpha, &  \frac{\partial}{\partial t}(\omega_2\wedge\alpha)&=-d\omega_3,\\
 \frac{\partial}{\partial t}\upsilon&=-\omega_1\wedge d\alpha,&\frac{\partial}{\partial t}(\omega_3\wedge\alpha)&=d\omega_2.
\end{align*}
These equations are equivalent to the evolution equations \eqref{eqn:hypoevolution} together with  $\upsilon=\frac12\omega_1^2$.
\end{proof}

\begin{remark}
The conditions of Proposition~\ref{prop:SL2CAndSU2} can be used to characterize the points of $\mathcal{H}$ that define an $\SU(2)$-structure. These points constitute what may be considered the space of deformations of the starting hypo structure. One can show directly that  the vector field $X_H$ is tangent to this space of deformations (see \cite{thesis}). In light of Theorem~\ref{thm:hamiltonian}, this can also be viewed as a consequence of the existence of solutions of  \eqref{eqn:hypoevolution}.
\end{remark}

\section{Evolution by gauge transformations}
\label{sec:gauge}
By definition, a solution of the hypo evolution equations is a one-parameter family of hypo structures satisfying \eqref{eqn:hypoevolution}. On the other hand, a one-parameter family $(\alpha(t),\omega_i(t))$ could satisfy \eqref{eqn:hypoevolution} without defining an $\SU(2)$-structure for all $t$. The condition of defining an $\SU(2)$-structure is only preserved infinitesimally by the evolution flow; by the non-uniqueness of the solutions of ODE's in a Fr\'echet space, this means that the condition of defining an $\SU(2)$-structure for all $t$ is not automatic. On the other hand, closedness of $\omega_1$, $\psi_2$, and $\psi_3$ is preserved in time by the evolution equations.

In this section we give an alternative description of the hypo evolution flow, which has the ``dual`` property that the condition of defining an $\SU(2)$-structure is automatically preserved in time, but the closedness of $\omega_1$, $\psi_2$, and $\psi_3$ is only preserved infinitesimally. This description will play a key r\^ole in the explicit calculations of Section~\ref{sec:integrating}.

The idea, borrowed from \cite{Stock}, is to restate the evolution equations in terms of gauge transformations. A gauge transformation on a $5$-manifold $M$ is by definition a $\GL(5,\R)$-equivariant map $s\colon F\to\GL(5,\R)$, where $F$ is the bundle of frames and $\GL(5,\R)$ acts on itself by the adjoint action. A gauge transformation $s$ defines a $\GL(5,\R)$-equivariant map from $F$ to itself by
\[u\to u s(u).\]
Accordingly, a gauge transformation acts on every associated bundle \mbox{$F\times_{\GL(5,\R)}V$} by
\[s\cdot [u,v]=[us(u),v].\]
Given an $\SU(2)$-structure $P\subset F$, one can define its intrinsic torsion as a map $P\to\su(2)^\perp\otimes\R^5$; the hypo condition implies that the intrinsic torsion takes values in a submodule isomorphic to $\Sym(\R^5)$. It is not surprising that the intrinsic torsion as a map
\[P\to \Sym(\R^5)\subset\gl(5,\R)\]
defines an ``infinitesimal gauge transformation'' that determines the evolution flow. The aim of this section is to define explicitly the infinitesimal gauge transformation, in the guise of an equivariant map
\[Q_P\colon F\to\gl(5,\R),\]
and prove the following:
\begin{theorem}
\label{thm:gauge}
Let $F$ be the bundle of frames on a $5$-manifold, and let $P_t\subset F$ be a one-parameter family of hypo $\SU(2)$-structures. Then $P_t$ satisfies the hypo evolution equations \eqref{eqn:hypoevolution} if and only if $P_t$ is obtained from a one-parameter family of gauge transformation
\[s_t\colon F\to\GL(5,\R), \quad s_0\equiv\id,\]
by
\[P_t=\{u s_t(u)\mid u\in P_0\},\]
and
\[s'_t s_t^{-1} =-Q_{P_t},\]
where juxtaposition represents matrix multiplication.
\end{theorem}
The minus sign appearing in the statement is a consequence of an arbitrary choice in the  definition of $Q_P$, motivated by the fact that in later sections we shall work with coframes rather than frames.

As a first step, we use the language of Section~\ref{sec:hamiltonian} to express the time derivative of the defining forms in terms of the intrinsic torsion. Recall from \cite{ContiSalamon} that an hypo structure $(\alpha,\omega_i)$ satisfies the following  ``structure equations'':
\begin{equation}
 \label{eqn:hypoIT}
 \left\{
\begin{aligned}
d\alpha&=\alpha\wedge\beta+f\omega_1+\omega^-,\\
d\omega_2&=\beta\wedge\omega_2+g\alpha\wedge\omega_3+\alpha\wedge\sigma_2^-,\\
d\omega_3&=\beta\wedge\omega_3-g\alpha\wedge\omega_2+\alpha\wedge\sigma_3^-.
\end{aligned}
\right.
\end{equation}
Here, $\beta$ is a $1$-form, $f$, $g$ are functions and $\omega^-$, $\sigma_2^-$, $\sigma_3^-$ are $2$-forms in $\Lambda^2_-(\ker\alpha)$. These functions and forms define the intrinsic torsion of the hypo $\SU(2)$-structure. These components can be defined for  generic $\SU(2)$-structures $(\alpha,\omega_i)$ as follows:
\begin{align*}
d\alpha&=\alpha\wedge\beta+ f\omega_1+f_2\omega_2+f_3\omega_3+\omega^-,\\
d\omega_1&=\gamma_1\wedge\omega_1+ \alpha\wedge(\lambda \omega_1 - g_2\omega_3 + g_3\omega_2+\sigma_1^-),\\
d\omega_2&=\gamma_2\wedge\omega_2+ \alpha\wedge(\lambda \omega_2 - g_3\omega_1 + g\omega_3+\sigma_2^-),\\
d\omega_3&=\gamma_3\wedge\omega_3+ \alpha\wedge(\lambda \omega_3 - g\omega_2 + g_2\omega_1+\sigma_3^-).
\end{align*}
An $\SU(2)$-structure also defines almost-complex structures $J_1,J_2,J_3$ on the distribution $\ker\alpha$, given explicitly by
\[\gamma\wedge\omega_j=(J_i\gamma) \wedge\omega_k, \quad Y\hook\omega_j= (J_iY)\hook\omega_k,\]
where $\gamma$ is a $1$-form orthogonal to $\alpha$, $Y$ is a vector in $\ker\alpha$, and $\{i,j,k\}$ is an even permutation of $\{1,2,3\}$.

We can now prove:
\begin{lemma}
\label{lemma:dinnerparty}
Let $(\omega_1(t),\psi_2(t),\psi_3(t),\upsilon(t))$ be an integral curve of the skew gradient $X_H$, and suppose it defines a hypo structure $(\alpha,\omega_i)$ at time zero. If  $\alpha_2(t)$, $\alpha_3(t)$, $\omega_2(t)$, $\omega_3(t)$ are defined as in Lemma~\ref{lemma:apointofH}, their derivatives depend on $(\alpha,\omega_i)$ and its intrinsic torsion as follows:
\begin{gather*}
\frac{d}{dt}\alpha_2|_{t=0}=\frac{d}{dt}\alpha_3|_{t=0}=(f+g)\alpha+J_1\beta,\\
\frac{d}{dt}\omega_2|_{t=0}=-f\omega_2+J_3\beta\wedge\alpha -\sigma_3^-,\quad
\frac{d}{dt}\omega_3|_{t=0}=-f\omega_3-J_2\beta\wedge\alpha +\sigma_2^-.
\end{gather*}
\end{lemma}
\begin{proof}
We work at $t=0$; hence, we can fix the metric underlying the hypo structure $(\alpha,\omega_i)$, whose volume form is $\frac12\alpha\wedge\omega_1^2$. Then under the identification $TM=T^*M$ the operators $A$, $A^*$ have the form
\[A\gamma= *\gamma \otimes \frac12\alpha\wedge\omega_1^2, \quad A^*\gamma=*\gamma \otimes (\frac12\alpha\wedge\omega_1^2)^{-1}.\]
Note that $V(\omega_1,\psi_2)$ is twice our fixed volume form.

By \eqref{eqn:X_H}, the time derivative of $\alpha_{\psi_3}$ is
\[\alpha_{\psi_3}'=2A^*(A\psi_3\wedge A\psi_3')=2A^*(A\psi_3\wedge A d\omega_2).\]
On the other hand
\[*d\omega_2=J_2\beta\wedge\alpha + g\omega_3 - \sigma_2^-, \quad *\psi_3=\omega_3,\]
hence
\[\alpha_{\psi_3}'=*2(\omega_3\wedge  J_2\beta\wedge\alpha + g\omega_3^2)\otimes \left(\frac12\alpha\wedge\omega_1^2\right)=(J_1\beta +2g\alpha)\otimes (\alpha\wedge\omega_1^2). \]
It follows from Lemma~\ref{lemma:dV} and \eqref{eqn:X_H} that
\[V(\psi_3,\upsilon)'=\omega_3\wedge\psi_3'+\alpha_3\wedge\upsilon'=\omega_3\wedge d\omega_2 -\alpha_3\wedge\omega_1\wedge d\alpha_2=(-f+g)V(\psi_3,\upsilon).\]
By definition \[\alpha_3(t)=V(\psi_3,\upsilon)^{-1}\alpha_{\psi_3},\] so
\[\alpha_3'=-\frac1{V^2(\psi_3,\upsilon)}V(\psi_3,\upsilon)'\alpha_{\psi_3} + V^{-1}(\psi_3,\upsilon)\alpha_{\psi_3}'=(f-g)\alpha_3+J_1\beta +2g\alpha_3. \]
Similarly, we obtain
\[\alpha_{\psi_2}'=-*2(\omega_2\wedge  J_3\beta\wedge\alpha - g\omega_2^2)=2J_1\beta +4g\alpha. \]
Notice also that
\[V(\omega_1,\psi_2)'=V(\psi_3,\upsilon)',\]
because the Hamiltonian $H$ is constant along integral lines. Thus $\alpha_2'=\alpha_3'$.

\smallskip
Now set \[X(t)=V^{-1}(\omega_1,\psi_2)X_{\omega_1};\]
we claim that
\begin{equation}
 \label{eqn:ddtX}
X'=(-f-g)X-(J_1\beta)^\sharp.
\end{equation}
Indeed, working again at $t=0$,
\begin{multline*}
X_{\omega_1}'=2A(\omega_1\wedge\omega_1')=-2A(\omega_1\wedge d\alpha_3)=-2A(\omega_1\wedge\alpha\wedge\beta+f\omega_1^2)\\
=(-J_1\beta-2f\alpha)^\sharp \otimes (\alpha\wedge\omega_1^2),
 \end{multline*}
whence
\[X'=-\frac1{V(\omega_1,\psi_2)^2}V(\psi_3,\upsilon)'X_{\omega_1} + V(\omega_1,\psi_2)^{-1}X_{\omega_1}'=(-f-g)X-(J_1\beta)^\sharp,\]
proving \eqref{eqn:ddtX}.

We can now compute $\omega_2'$ by
\[(X\hook\psi_2)'=(-f-g)\omega_2-J_1\beta\hook\psi_2 + g\omega_2-\sigma_3^-=-f\omega_2+J_3\beta\wedge\alpha -\sigma_3^-.\]
By definition, to compute $\omega_3$ we should take the interior product with $V^{-1}(\psi_3,\upsilon)X_\upsilon$ rather than $X(t)$; however, it is easy to verify that the vector fields coincide up to first order at time $0$. Thus
\[(X\hook\psi_3)'=(-f-g)\omega_3-J_1\beta\hook\psi_3 + g\omega_3+\sigma_2^-=-f\omega_3-J_2\beta\wedge\alpha +\sigma_2^-.\qedhere\]
\end{proof}
The second step is to define $Q_P$. Let $P$ be an $\SU(2)$-structure and $\pi\colon P\to M$ the projection. The intrinsic torsion defines global differential forms on $M$, which can be pulled back to basic forms on $P$; such forms belong to the algebra generated by the components $\theta^1,\dotsc, \theta^5$ of the tautological form. Thus, we obtain functions on $P$ determined by
\[\pi^*\omega^- = 2\omega^-_a (\theta^{12}-\theta^{34})+2\omega^-_b (\theta^{13}-\theta^{42})+2\omega^-_c (\theta^{14}-\theta^{23});\]
in a similar way, we define functions
\[(\sigma^-_k)_a,(\sigma^-_k)_b,(\sigma^-_k)_c\colon P\to\R, \quad k=2,3.\]
Likewise a one-form orthogonal to $\alpha$, such as $\beta$, determines a function $P\to\R^4$ by
\[\pi^*\beta=\beta_1\theta^1+\dotsc+\beta_4\theta^4.\]
We can now define $Q_P$ in terms of the ``hypo'' part of the intrinsic torsion, i.e. the components appearing in \eqref{eqn:hypoIT}.
\begin{proposition}
Given an $\SU(2)$-structure $P$ on $M$, at each $u\in P$ let $\tilde Q_P(u)$ be the symmetric matrix whose upper triangular part is
\[
 \begin{pmatrix}
 (\sigma_2^-)_c -(\sigma_3^-)_b - \omega^-_a  & (\sigma_2^-)_b + (\sigma_3^-)_c  & -(\sigma_2^-)_a - \omega^-_c & -(\sigma_3^-)_a+ \omega^-_b \\
   &-(\sigma_2^-)_c +(\sigma_3^-)_b - \omega^-_a & -(\sigma_3^-)_a - \omega^-_b & (\sigma_2^-)_a - \omega^-_c\\
   & & -(\sigma_2^-)_c-(\sigma_3^-)_b + \omega^-_a & (\sigma_2^-)_b - (\sigma_3^-)_c\\
   & & & (\sigma_2^-)_c+(\sigma_3^-)_b + \omega^-_a
 \end{pmatrix}.
\]
Define $Q_P\colon P\to\gl(5,\R)$ by
\[
 Q_P(u)=\begin{pmatrix}
-\frac f2\id +\tilde Q_P     & J_1\beta \\
(J_1\beta)^T & f+g\\
 \end{pmatrix}.
\]
Then $Q_P$ is $\SU(2)$-equivariant, and therefore defines a section of
\[P\times_{\SU(2)}\gl(5,\R)\cong\End(TM).\]
 \end{proposition}
\begin{proof}
From the general theory, the intrinsic torsion map $P\to \R^5\otimes\su(2)^\perp$ is $\SU(2)$-equivariant. The definition of $Q$ amounts to composing the intrinsic torsion with a map $\R^5\otimes\su(2)^\perp\to \gl(5,R)$, which must be checked to be equivariant.

The non-trivial submodules of $\R^5\otimes\su(2)^\perp$ are isomorphic to the space $\Lambda^2_-\R^4$ spanned by
\begin{equation}
 \label{eqn:lambda2minus}
e^{12}-e^{34},e^{13}-e^{42},e^{14}-e^{23}.
\end{equation}
Denote by $e^k\odot e^h$ the element $e^h\otimes e^k+e^k\otimes e^h$ of $\gl(5,\R)\cong \R^5\otimes\R^5$. It is easy to check that $\SU(2)$ acts on the bases
\begin{gather*}
-e^1\otimes e^1-e^2\otimes e^2+e^3\otimes e^3+e^4\otimes e^4, \quad e^1\odot e^4-e^2\odot e^3, \quad -e^1\odot e^3-e^2\odot e^4,\\
e^2\odot e^4-e^1\odot e^3, \quad e^1\odot e^2+e^3\odot e^4,\quad e^1\otimes e^1-e^2\otimes e^2-e^3\otimes e^3+e^4\otimes e^4,\\
e^2\odot e^3+e^1\odot e^4,\quad e^1\otimes e^1-e^2\otimes e^2+e^3\otimes e^3-e^4\otimes e^4,  \quad -e^1\odot e^2+e^3\odot e^4,
\end{gather*}
as it acts on the basis \eqref{eqn:lambda2minus}.

Similarly, $e^5\otimes e^5$ is invariant under $\SU(2)$, and the map
\[\Span{e^1,\dotsc e^4}\to \gl(5,\R), \quad e^i\to e^i\odot e^5\]
is $\SU(2)$-equivariant.
\end{proof}
The map  $Q_P$ will be extended equivariantly to a map $Q_P\colon F\to\gl(5,\R)$.

\begin{proof}[Proof of Theorem~\ref{thm:gauge}]

We work at a point $x\in M$. Let $u\in P_x$. The map $Q_{P_0}$ was constructed in such a way that
\begin{align*}
[u,Q_{P_0}(u)e^5]&=-(f+g)\alpha-J_1\beta,\\
[u,Q_{P_0}(u)(e^{12}+e^{34})]&=\alpha\wedge\beta+f\omega_1+\omega^-,\\
[u,Q_{P_0}(u)(e^{13}+e^{42})]&=f\omega_2-J_3\beta\wedge\alpha +\sigma_3^-, \\
[u,Q_{P_0}(u)(e^{14}+e^{23})]&=f\omega_3+J_2\beta\wedge\alpha -\sigma_2^-.
\end{align*}
At generic $t$, these equations take the form
\[-(f+g)\alpha-J_1\beta= [u s_t(u), Q_{P_t}(u s_t(u)) e^5] =  [u, Q_{P_t}(u) (s_t(u) e^5)],\]
and so on.

Now suppose that $P_t$ is defined by $s_t$ as in the statement; we must show that the hypo evolution equations are satisfied. By construction
\[\alpha(t)=s_t\cdot \alpha(0), \quad \omega_i(t)=s_t\cdot \omega_i(0).\]
Thus
\[\alpha'(t)_x = [u,s_t(u)'e^5)]=[u,-Q_{P_t}(u)s_t(u)e^5];\]
this holds for all $u$, hence
\[\alpha'= (f+g)\alpha+J_1\beta.\]
Similarly, we find
\begin{gather*}
\frac{d}{dt}\omega_1=-\alpha\wedge\beta-f\omega_1-\omega^-,\quad
\frac{d}{dt}\omega_2=-f\omega_2+J_3\beta\wedge\alpha -\sigma_3^-,\\
\frac{d}{dt}\omega_3=-f\omega_3-J_2\beta\wedge\alpha +\sigma_2^-.
\end{gather*}
These equations are consistent with Lemma \ref{lemma:dinnerparty}, and they imply the hypo evolution equations.
\end{proof}

\begin{remark}
By construction the metric $g_t$ varies accordingly to
\[g_t'(X,Y)=g_t(Q_{P_t}X,Y)+g_t(X,Q_{P_t}Y)=2g_t(Q_{P_t}X,Y);\]
on the other hand it follows easily from the Koszul formula that with respect to the generalized cylinder metric $g_t+dt^2$,
\[g_t'(X,Y)=-2g_t(W(X),Y).\]
where $W$ is the Weingarten tensor of the hypersurface $N\times\{t\}$, with orientations chosen so that $\frac{\partial}{\partial t}$ is the positively oriented normal. Thus, $Q_{P_t}$ equals minus this Weingarten tensor.
\end{remark}

\section{A model for evolving Lie algebras}
\label{sec:D}
Given a hypo Lie algebra $\lie{g}$, one way of expressing a solution of the evolution equations is by giving a one-parameter family of coframes on $\lie{g}$. The natural setting to study this problem is therefore the category of Lie algebras with a fixed coframe. In this section we introduce this category, and also a model for it, in the guise of a discrete category over a real affine variety. Computations are considerably easier on this model category $\mathcal{D}$, since the group of automorphisms is effectively factored out; for this reason, the classification results of this paper will be formulated in terms of $\mathcal{D}$. We shall fix the dimension to five for definiteness, and in view of the applications to follow.

Let $(D)$ be the category whose objects are pairs $(\lie{g},u)$, with $\lie{g}$ a $5$-dimensional Lie algebra and $u\colon\R^5\to\lie{g}^*$ a linear isomorphism; $u$ induces naturally another isomorphism
\[u\colon\Lambda^2\R^5\to\Lambda^2\lie{g}^*, \quad u(\alpha\wedge\beta)=u(\alpha)\wedge u(\beta).\]
We define
\[\Hom_{(D)}((\lie{g},u),(\lie{h},v))=\{f\in\Hom(\lie{g},\lie{h})\mid u=f^t\circ v\},\]
where $\Hom(\lie{g},\lie{h})$ is the space of Lie algebra homomorphisms (though $f$ is forced to be an isomorphism).
There is a natural right action of $\GL(5,\R)$ on $(D)$. This means that each $g\in\GL(5,R)$ defines a functor
\[J_g(\lie{g},u) = (\lie{g},u\circ g), \quad J_g(f)=f,\]
and $J_g\circ J_h = J_{hg}$.

A model for $(D)$ is given by the set
\[\mathcal{D}=\{d\in\Hom(\R^5,\Lambda^2\R^5)\mid \hat d\circ d=0\},\]
where $\hat d$ is the derivation $\Lambda\R^5\to\Lambda\R^5$ induced by $d$.
There is a natural right action of $\GL(5,\R)$ on $\mathcal{D}$, namely
\[(\mu(g)d)(\beta)=g^{-1}d(g\beta), \quad g\in\GL(5,\R),d\in\mathcal{D}.\]
We shall also denote by $\mathcal{D}$ the discrete category on $\mathcal{D}$; this means that the objects are the points of $\mathcal{D}$, and the only morphisms are the  identity morphisms. There is a $\GL(5,\R)$-equivariant functor
\[F\colon(D)\to\mathcal{D}, \quad F(\lie{g},u)=u^{-1}\circ d_{\lie{g}}\circ u, \quad F(f)=\id,\]
where $d_{\lie{g}}\colon\lie{g}^*\to\Lambda^2\lie{g}^*$ is the Chevalley-Eilenberg differential.

We claim that $F$ is an equivalence, meaning that $F$ maps isomorphically morphisms onto morphisms, and every object of $\mathcal{D}$ is isomorphic to some $F(\lie{g},u)$.
\begin{lemma}
The functor $F\colon(D)\to\mathcal{D}$ is a $\GL(5,\R)$-equivariant equivalence.
\end{lemma}
\begin{proof}
First, we must prove that $F$ is well defined. This is because if $f$ is in $\Hom(\lie{g},u),(\lie{h},v))$, then $u=f^t\circ v$, so
\[F(\lie{g},u)=u^{-1}d_{\lie{g}} u = v^{-1}\circ (f^t)^{-1}\circ d_{\lie{g}} f^t\circ v = v^{-1}d_{\lie{h}}\circ v=F(\lie{h},v).\]
Equivariance is proved similarly. The condition on morphisms is trivial. Finally, given $d$ in $\mathcal{D}$, $\R^5$ has an induced Lie algebra structure such that $F(\R^5,\id)$ is $d$ itself.
\end{proof}
The space $\mathcal{D}$, whilst not mysterious in itself, is endowed with additional structure by the functor $F$, as becomes clear in the following lemma.
\begin{lemma}
\label{lemma:inducedvector}
Let $A(t)$ be a one-parameter family in $\gl(5,\R)$, and let $(\lie{g},u(t))$ be a one-parameter family in $(D)$ satisfying the differential equation
\[u'(t)=u(t)\circ A^T(t).\]
Then the basis $e^i(t)=u(t)(e_i)$ satisfies the differential equation
\begin{equation}
 \label{eqn:Aprimee}
\begin{pmatrix} (e^1)'\\ \dots\\ (e^5)'\end{pmatrix} = A(t)\begin{pmatrix} e^1\\ \dots\\ e^5\end{pmatrix}
\end{equation}
and the curve in $\mathcal{D}$ given by $d(t)=F(\lie{g},u(t))$ satisfies
\[d'(t)=\mu_{*e}(A^T(t))d(t).\]
\end{lemma}
\begin{proof}
Equation~\eqref{eqn:Aprimee} follows from the fact that
\[\begin{pmatrix} e^1\\ \dots\\ e^5\end{pmatrix}\]
can be identified with $u^T$.

For the second part of the statement, let $u(t)=u(0)g(t)$, with $g(t)$ a curve in $\GL(5,\R)$. Then
\[F(\lie{g},u(t))=g^{-1}(t)F(\lie{g},u(0))g(t)=\mu(g(t))F(\lie{g},u(0)),\]
i.e. $d(t)=\mu(g(t))d(0)$.
By construction
$u'(t)=u(0)g'(t)$
so
\[g'(t)=g(t)A^T(t).\]
In order to compute $d'(t)$ at $t=t_0$, set $g=g(t_0)$. Then
\[d(t)=\mu(g^{-1}g(t))d(t_0),\]
so
\[d'(t_0)=\mu_{*e}(A^T(t_0))d(t_0).\qedhere\]
\end{proof}
In the following definition, we give $\gl(5,\R)$ the structure of a discrete category.
\begin{definition}
An \dfn{infinitesimal gauge transformation} on $(D)$ is a functor $X\colon (D)\to \gl(5,\R)$ such that the functor $\hat X$ induced by the diagram
\[\xymatrix{ (D) \ar[dr]^X\ar[d]^F \\ \mathcal{D}\ar[r]^{\hat X} &\gl(5,\R)}\]
is a smooth map.
\end{definition}
The functoriality guarantees that $X_{(\lie{g},u)}=X_{(\lie{g},f^t\circ u)}$ whenever $f\colon\lie{g}\to\lie{g}$ is an isomorphism. This means that $X_{(\lie{g},u)}$ is unaffected when $u$ is acted upon by an automorphism of $\lie{g}$. This invariance is factored out when we pass to $\hat X$.

\begin{proposition}
\label{prop:inducedvectorfield}
An infinitesimal gauge transformation $X$ on $(D)$ induces a vector field $\tilde X$ on $\mathcal{D}$ by
\[\tilde X_d=\mu_{*e}(\hat X^T_{d})d.\]
If $(\lie{g},u(t))$ is a one-parameter family in $(D)$ satisfying the differential equation
\[u'(t)=u(t)\circ X^T_{u(t)},\]
then $d(t)=F(\lie{g},u(t))$ is an orbit of $\tilde X$.
\end{proposition}
\begin{proof}
It suffices to set $A(t)=X_{u(t)}$ in Lemma~\ref{lemma:inducedvector}.
\end{proof}
Replacing $X$ with $\hat X$ behaves well with respect to equivariance. Indeed, consider the natural action of $\GL(5,\R)$ on vector fields on $\mathcal{D}$
\[(g\cdot Y)_x=\mu(g)Y_{\mu(g^{-1})x}.\]
\begin{lemma}
\label{lemma:invariance}
Let $X$ be an infinitesimal gauge transformation on $(D)$. Suppose that $\hat X\colon\mathcal{D}\to\gl(5,\R)$ is equivariant under some $G\subset\GL(5,\R)$, i.e.
\[\hat X_{\mu(g)x}=\Ad(g^{-1})\hat X_x, \quad g\in G.\]
Then $\tilde X$ is $G$-invariant.
\end{lemma}
\begin{proof}
By hypothesis
\begin{multline*}
\tilde X_{\mu(g)d}=\mu_{*e}(\hat X^T_{\mu(g)d})\mu(g)d=\mu_{*e}(\Ad(g^{-1})(\hat X^T_{d}))d=(\Ad(\mu(g))\mu_{*e}\hat X^T_{d})\mu(g)d\\
=\mu(g)\mu_{*e}\hat X^T_{d}\mu(g^{-1})\mu(g)d=\mu(g)\mu_{*e}\hat X^T_{d}d=(g\cdot \tilde X)_{\mu(g)d}.\qedhere
\end{multline*}
\end{proof}

We can now restate Theorem~\ref{thm:gauge} in the language of this section, introducing the infinitesimal gauge transformation on $(D)$ given by
\[X(\lie{g},u)=Q_P(v),\]
where $P=v\SU(2)$ is the $\SU(2)$-structure on $\lie{g}$ defined by $v=(u^{T})^{-1}$.
\begin{proposition}
$X$ is an infinitesimal gauge transformation on $(D)$, and the vector field $\tilde X$ induced on $\mathcal{D}$ is $\LieG{U}(2)$-invariant. Given a Lie algebra $\lie{g}$ and a one-parameter family of hypo structures $(\alpha(t),\omega_i(t))$ on $\lie{g}$ that satisfies the hypo evolution equations, there exists a one-parameter family of coframes $u(t)\colon\R^5\to\lie{g}^*$ such that
\begin{itemize}
\item $u(t)$ is adapted to $(\alpha(t),\omega_i(t))$;
\item the curve $d(t)=F(\lie{g},u(t))$ is an orbit of $\tilde X$ in $\mathcal{D}$;
\item $u'(t)=u(t)\circ \hat X_{d(t)}$.
\end{itemize}
\end{proposition}
\begin{proof}
The intrinsic torsion is invariant under automorphisms of $\lie{g}$, so $X$ is indeed an infinitesimal gauge transformation. Invariance follows from Lemma~\ref{lemma:invariance} and the fact that $\hat X$ factors through the intrinsic torsion map, which is linear and equivariant.

Let $v(0)=(u(0)^{-1})^T$. Theorem~\ref{thm:gauge} gives a one-parameter family of adapted frames
\[v(t)=v(0)s_t, \quad s_t's_t^{-1}=-Q_{P_t}(v(0)),\]
which in turn determines a one-parameter family of coframes
\[u(t)=u(0)(s_t^{-1})^T;\]
then
\[Q_{P_t}(v(0))=Q_{P_t}(v(t)s_t^{-1})=\Ad_{s_t}(Q_{P_t}(v(t))) =-\Ad_{s_t} X(\lie{g},u(t)),\]
so $s_t^{-1}s_t'=X(\lie{g},u(t))$ and
\[u'(t)=u(t)\circ X^T(\lie{g},u(t)).\]
Now, by Proposition~\ref{prop:inducedvectorfield}, $d(t)$ is an orbit of the induced vector field $\tilde X$.
Finally, $X^T(\lie{g},u(t))=\hat X_{d(t)}$ holds by definition and because $X$ is symmetric.
\end{proof}
\section{Hypo nilmanifolds}
There are exactly nine real nilpotent Lie algebras of dimension five, classified in \cite{Magnin}. It was proved in \cite{ContiSalamon} that only six out of the nine carry a hypo structure. In this section, we classify the hypo structures on these six Lie algebras in terms of the space $\mathcal{D}$ introduced in Section~\ref{sec:D}.

First we need a definition. We say $d$ in $\mathcal{D}$ is \dfn{hypo} if
\[d(e^{12}+e^{34})=0, \quad d(e^{135}+e^{425})=0, \quad d(e^{145}+e^{235})=0.\]
This means that if $d=F(\lie{g},u)$, the $\SU(2)$-structure on $\lie{g}$ determined by the coframe $u$ (see Section~\ref{sec:SU2}) is hypo in the sense that the defining forms $\omega_1$, $\psi_2$ and $\psi_3$ are closed under the Chevalley-Eilenberg operator $d_{\lie{g}}$. Then any Lie group $G$ with Lie algebra $\lie{g}$ has an induced left-invariant hypo structure. Moreover, if $\Gamma$ is a discrete uniform subgroup, a hypo structure is induced on the compact quotient $\Gamma\backslash G$. In the case that $\lie{g}$ is nilpotent, such a subgroup always exists.

Now the coframe $u$ determines the same $\SU(2)$-structures as the other coframes in its $\SU(2)$-orbit; so, isomorphism classes of hypo Lie algebras are elements of
\[\{d\in\mathcal{D}\mid d \text{ is hypo}\}/\SU(2).\]
On the other hand, the space of $\SU(2)$-structures on an isomorphism class of Lie algebras can be expressed as the biquotient
\[(\GL(5,\R)d)/\SU(2) \cong \Aut(\lie{g})\backslash\Iso(\R^5,\lie{g}^*)/\SU(2),\]
where $d=F(\lie{g},u)$ for some  representative $\lie{g}$ and some coframe $u$.

The hypo equations are actually invariant under a larger group than $\SU(2)$, namely $U(2)\times\R^*\times\R^*\times\Z_2$. From the point of view of the evolution equations, it is more natural to consider the quotient by $U(2)$, which is the largest subgroup that leaves the flow invariant (although the diagonal $\R^*$ acting by scalar multiplication only affects the flow by a change of time scale). Thus, for all $d$ in $\mathcal{D}$ we define
\[H_d=\{gd\mid g\in\GL(5,\R), gd \text{ is hypo}\}\!/\raisebox{-.7ex}{\ensuremath{\LieG{U}(2)}}.\]
In this section we compute $H_d$ for all five-dimensional nilpotent Lie algebras. In fact, we show that these Lie algebras come in a hierarchy, with exactly three Lie algebras lying at its top level. It will follow that that the study of hypo nilpotent Lie algebras is reduced to the study of three families of hypo Lie algebras.

The first family corresponds generically to $(0,0,12,13,14)$; it is defined as
\begin{gather*}
\mathcal{M}_1=\left\{d_{\lambda,\mu,h,k}\mid h,k,\lambda,\mu\in\R\right\},\\
d_{\lambda,\mu,h,k}=\left(\lambda e^{35},he^{35}+ke^{15}, 0,(-\lambda e^2+he^1+\mu e^3)\wedge e^5,0\right).
\end{gather*}
The second family, corresponding generically to  $(0,0,0,12,13+24)$, consists of elements of the form
\begin{multline}
\label{eqn:elementofM2}
d_{x,y,h,k,\lambda,\mu}=
\bigl(0, xe^{34}+\lambda e^{35},0, x(e^{14}-e^{23})-ye^{34}+\lambda e^{15}-\mu e^{35},\\
-h(e^{14}-e^{23})+ke^{34}-x e^{15}+y e^{35}\bigr).
\end{multline}
More precisely, this family is an algebraic subvariety of $\mathcal{D}$ given by
\[\mathcal{M}_2=\left\{d_{x,y,h,k,\lambda,\mu}\mid
\rk\begin{pmatrix} x&y&\lambda&\mu\\ h&k&x&y\end{pmatrix}<2\right\}.\]
The third family has some $\LieG{U}(2)$-orbits in common with $\mathcal{M}_2$; it is defined as
\[\mathcal{M}_3=\left\{d_{\lambda,\mu}\mid \lambda,\mu\in\R\right\},\quad d_{\lambda,\mu}=\left(0,0,0,0,(\lambda+\mu)e^{12}+(\lambda-\mu) e^{34})\right). \]
\begin{theorem}\label{thm:hierarchy}
An element of $\mathcal{D}$ is hypo if and only if its $\LieG{U}(2)$ orbit intersects $\mathcal{M}_1$, $\mathcal{M}_2$ or $\mathcal{M}_3$.
The isomorphism classes of non-abelian nilpotent Lie algebras of dimension five that admit a hypo structure can be arranged in the  diagram
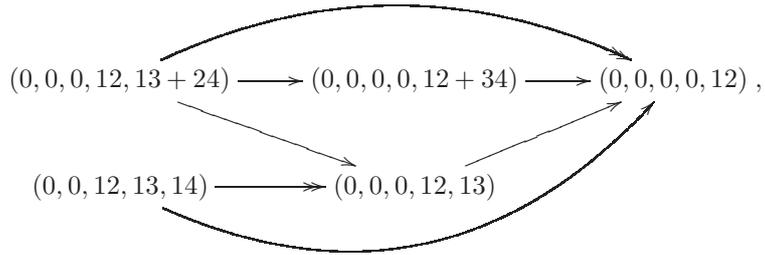
\begin{multline*}
\\
\xymatrix{
(0,0,0,12,13+24)\ar@{>>}@/^1cm/[rr]\ar@{>}[r]\ar[dr] &		(0,0,0,0,12+34)\ar[r] & (0,0,0,0,12)\\
(0,0,12,13,14)\ar@{>>}[r]\ar@/_1.5cm/[urr] &	(0,0,0,12,13)\ar[ur]
},\\
\end{multline*}
where $d_1$ is connected by an arrow to $d_2$ if the closure of the $\GL(5,\R)$-orbit of $d_1$ contains $d_2$, and the arrow has a double head if $\overline{H_{d_1}}$ contains $H_{d_2}$.
\end{theorem}
We already know from \cite{ContiSalamon} that the isomorphism classes for which $H_d$ is nonempty are exactly those appearing in the diagram. In order to prove the theorem, we must compute the space $H_d$ for $d$ in each of these isomorphism classes. The rest of this section consists in this computation, broken up into several lemmas. We use the following notation: if $V\subset\mathcal{D}$, then $V/U(2)$ denotes the image of $V$ under the projection $\mathcal{D}\to\mathcal{D}/U(2)$, regardless of whether $V$ is the union of orbits.

\begin{lemma}
\label{lemma:family1}
Let $d$ be in the $\GL(5,\R)$-orbit of $(0,0,12,13,14)$. Then
\[H_d = \{d_{\lambda,\mu,h,k}\in\mathcal{M}_1\mid \lambda>0,k\neq0\}\!/\raisebox{-.7ex}{\ensuremath{\LieG{U}(2)}};\]
in particular, the closure of $H_d$ is $\mathcal{M}_1/\LieG{U}(2)$.
\end{lemma}
\begin{proof}
As usual, we denote by $e^1,\dotsc, e^5$ the standard coframe on $\R^5$. By construction there is another coframe $\eta^1,\dotsc, \eta^5$ on $\R^5$ such that
\[d\eta^1=0, \quad d\eta^k=\eta^{1}\wedge\eta^{k-1}, k=2,\dotsc,5.\] If we set $V^k=\Span{\eta^1,\dotsc, \eta^k}$, it is not hard to check that the $V^i$ are independent of the choice of the coframe; for instance,
\[V^1=\{\beta\in(\R^5)^*\mid \beta\wedge d\gamma=0, \gamma\in (\R^5)^*\}.\]
The space of closed $3$-forms is given by
\begin{gather*}
Z^3=\Span{\eta^{123},
\eta^{124},
\eta^{125},
\eta^{134},
\eta^{135},
\eta^{145},
\eta^{234}}.
\end{gather*}
Now assume $d$ is hypo; we must determine $d$ up to $U(2)$ action. Arguing like in Proposition~7 of \cite{ContiFernandezSantisteban}, we can see that $e^5$ is in $V^1$. Indeed the space $Z^3\wedge V^1$ is one-dimensional, so $\psi_2\wedge V^1$ and $\psi_3\wedge V^1$ are linearly dependent, which is only possible if $e^5$ is in $V^1$.

In particular, $e^5$ is in $V^4$. Therefore, the span of $e^1,e^2,e^3,e^4$ intersects $V^4$ in a three-dimensional space;
up to an action of $\SU(2)$, we can assume that $e^1,e^2,e^3$ are in $V^4$ and $e^4$ is not. Thus $de^4=\gamma\wedge e^5$ and $\gamma$ is not in $V^3$. Then
\[0=d\omega_1=de^3\wedge e^4\mod \Lambda^3V^4,\]
so $de^3=0$. Up to the action of $U(2)$, we can rotate $e^1$ and $e^2$ and obtain that $e^1$ is in $V^3$ and $e^2$ is not.
Then
\[de^1=\lambda e^{35}, \quad de^2=he^{35}+ke^{15}, \quad k,\lambda\neq 0.\]
 Moreover
\[0=d\omega_1=de^{12}+e^{35}\wedge\gamma\]
implies that
\[\gamma=-\lambda e^2+h e^1+\mu e^3.\]
Thus $d$ has the form
\[(\lambda e^{35},he^{35}+ke^{15}, 0,(-\lambda e^2+he^1+\mu e^3)\wedge e^5,0), \quad \lambda,k\neq0.\]
The matrices in $\LieG{U}(2)$ that map one element of the above family to another element of the family constitute a group generated by
\begin{equation}
\label{eqn:subgroupU2}
\begin{pmatrix}-\id_2 \\& \id_2\\ &&1\end{pmatrix},\quad \begin{pmatrix}\id_2 \\& -\id_2\\ &&1\end{pmatrix}.
\end{equation}
The first two elements have the effect of changing the signs of $\lambda, h$, and the third element has no effect.
\end{proof}

\begin{lemma}
If $d$ is in the $\GL(5,\R)$-orbit of $(0,0,0,12,13+24)$, then
\[H_d = \{d_{x,y,h,k,\lambda,\mu}\in \mathcal{M}_2\mid h,\lambda>0\}\!/\raisebox{-.7ex}{\ensuremath{\LieG{U}(2)}};\]
in particular, the closure of $H_d$ is $\mathcal{M}_2/\LieG{U}(2)$.
\end{lemma}
\begin{proof}
Suppose $d$ is hypo. We define $k$-dimensional subspaces $V^k\subset(\R^5)^*$ by
\[V^4=\{\beta\in (\R^5)^*\mid (d\beta)^2=0\}, \quad V^3=\ker d, \quad V^2=\{\beta\in(\R^5)^*\mid \beta\wedge d(V^4)=0\}.\]
If $e^1,e^2,e^3,e^4$ are in $V^4$, then all the $\omega_i$ are closed, and their restrictions to the non-abelian subalgebra $\ker e^5$ determine a hyperk\"ahler structure, which is absurd.

So up to $\LieG{U}(2)$ action we can assume that $V^4$ contains $e^1,e^2,e^3$ but not $e^4$. Then
\[d\omega_1=de^3\wedge e^4\mod \Lambda^3V^4,\]
so $e^3$ is closed. Moreover up to $\LieG{U}(2)$ action we can assume that  $e^1$ is in $V^3$. There are two cases.

\emph{a}) Suppose $e^5$ is in $V^4$. Then
\[d\psi_2=de^{25}\wedge e^{4}\mod \Lambda^4V^4,\quad d\psi_3=-de^{15}\wedge e^{4}\mod \Lambda^4V^4,\]
so $e^{15}$ and $e^{25}$ are closed; since $e^3$ is also closed,
\[de^4\wedge e^{25}=0=de^4\wedge e^{15},\]
hence
\begin{equation}
\label{eqn:de4}
de^4\in\Span{e^{12},e^{15},e^{25},e^{35}}.
\end{equation}
Moreover
\[0=d\omega_1=e^1\wedge de^2+e^3\wedge de^4.\]
It follows that $e^5$ is in $V^3$: otherwise some linear combination $e^2+ae^5$ is in $V^3$, hence $de^2=-ade^5$; this is absurd by
\[0\neq de^4\wedge e^3=-de^2\wedge e^1=ade^5\wedge e^1=0.\]
So $V^3$ is spanned by $e^1,e^3,e^5$; since $e^{25}$ is closed, $de^2$ is a linear combination of $e^{15}$ and $e^{35}$. By \eqref{eqn:de4} and $(de^4)^2\neq0$, $de^4$ has a component along $e^{12}$; since by construction $de^2\wedge de^4=0$, it follows that $de^2$ is a multiple of $e^{15}$. But then $e^3\wedge de^4=0$, which is absurd.

\emph{b}) Suppose $e^5$ is not in $V^4$; then $e^4+ze^5$ is in $V^4$ for some nonzero $z$. Using the full group of symmetries of the hypo equations, we can rescale $e^5$, so that $z=1$. For brevity, we shall write $e^{4+5}$ for $e^4+e^5$. Then
\[
0=d\psi_2=d((e^{13}+e^{4+5,2})\wedge e^5)= de^{4+5,2}\wedge e^5 \mod \Lambda^4V^4.\]
Closedness of $e^{2,4+5}$ and the fact that $e^1,e^3$ are in $V^3$ imply that
$de^2,de^{4+5}$ are in
\[\Span{e^{12},e^{1,4+5},e^{23},e^{3,4+5}},\]
and more precisely they have the form
\[de^2=xe^{12}+ye^{1,4+5}+he^{23}+ke^{3,4+5}, \quad
de^{4+5}=a e^{12}- x e^{1,4+5}+be^{23}+he^{3,4+5}.\]
Moreover by $d(e^{1,4+5}+e^{23})=0$ we see that $h=-y$, $b=x$, so
\[de^2=xe^{12}+ye^{1,4+5}-ye^{23}+ke^{3,4+5}\]
\[de^{4+5}=a e^{12}- xe^{1,4+5}+xe^{23}-y e^{3,4+5}.\]
We know that $de^2$ and $de^{4+5}$ are linearly dependent; by writing components and computing determinants, we deduce
\begin{equation}
 \label{eqn:dependence}
 x^2=-ay, \quad xy=-ak, \quad y^2=kx.
\end{equation}
On the other hand
\begin{equation}
\label{eqn:boh}
0=e^1\wedge de^2+e^3\wedge de^4=-ye^{123}+ke^{13,4+5}+a e^{123}+ x e^{13,4+5}-e^3\wedge de^5,
\end{equation}
whence
\[e^{13}\wedge de^5=0.\]
Now
\begin{align*}
0&=d\psi_2=(e^{13}+e^{4+5,2})\wedge de^5=e^{4+5,2}\wedge de^5,\\
0&=d\psi_3=(e^{1,4+5}+e^{23})\wedge de^5.
\end{align*}
Therefore
\[de^5\in\Span{e^{12},e^{1,4+5}-e^{23},e^{3,4+5}},\]
and more precisely, using \eqref{eqn:boh}
\[de^5=(a-y)e^{12}-(k+x)(e^{1,4+5}-e^{23}) + \lambda e^{3,4+5}.\]
Then
\[d^2e^5=((a-y)y+2x(k+x)-\lambda a)e^{123} + (-(a-y)k-\lambda x-2y(k+x)) e^{13(4+5)}.\]
whence
\[\lambda a = ay-y^2+2kx+2x^2, \quad \lambda x= -ak +ky-2ky-2xy.\]
Using \eqref{eqn:dependence},
we get
\[\lambda a = y^2+x^2, \quad \lambda x= -xy -ky\]

We claim that $x$ and $y$ are zero. Indeed, suppose otherwise; by \eqref{eqn:dependence}, both $x$ and $y$ are nonzero and
\[a=-\frac{x^2}y, \quad k=\frac{y^2}x.\]
This implies that $(de^5)^2$ is zero, which is absurd.

So $x=y=0$. Then $ak=0=a\lambda$. Then  $(de^5)^2\neq0$ implies $k\neq0$, so $a=0$ and
\[de^2=ke^3(e^4+e^5), \quad de^{4+5}=0, \quad de^5=-k(e^{1(4+5)}-e^{23}) + \lambda e^{3(4+5)}.\]
These equations were obtained by declaring $z=1$. In the general case (changing names to the variables),
\[d_{x,y,z}=\left(0, xe^3(e^4+ze^5),0,(xe^{1}- ye^{3})(e^4+ze^5)-xe^{23}, (-xe^{1}+ ye^{3})(\frac1ze^4+e^5)+\frac{x}{z}e^{23}\right).\]
This is not a closed set in $\mathcal{D}$. To compute the closure, we set
\[\lambda = xz, \quad \mu=yz, \quad h=\frac xz, \quad k=\frac yz;\]
these variables must satisfy
\[\operatorname{rk}\begin{pmatrix} \lambda& \mu & x& y\\ x&y&h&k\end{pmatrix}<2.\]
Then $d$ has the form \eqref{eqn:elementofM2}; the corresponding Lie algebra is isomorphic to $(0,0,0,12,13+24)$ if and only if both $h$ and $\lambda$ are nonzero.

Like in Lemma~\ref{lemma:family1}, the elements of $\LieG{U}(2)$ that map $d$ to an element of the same family \eqref{eqn:elementofM2} constitute a group generated by the matrices \eqref{eqn:subgroupU2}, which have the effect of changing the signs of $(x,h,\lambda)$ and $(y,h,\lambda)$ respectively; it follows that $h$ and $\lambda$ can be assumed to be positive.
\end{proof}

\begin{lemma}
If $d$ in $\mathcal{D}$ is in the $\GL(5,\R)$-orbit of $(0,0,0,0,12+34)$, then
\[H_d = \left\{d_{\lambda,\mu}\in\mathcal{M}_3\mid \lambda,\mu\in\R, 0\leq\mu\neq\abs{\lambda}\right\}\!/\raisebox{-.7ex}{\ensuremath{\LieG{U}(2)}}\;\]
in particular the closure of $H_d$ is $\mathcal{M}_3/\LieG{U}(2)$, and
\[\left\{d_{\lambda,\mu}\in\mathcal{M}_3\mid \mu^2-\lambda^2>0\right\}/\raisebox{-.7ex}{\ensuremath{\LieG{U}(2)}}\subset\mathcal{M}_2/\raisebox{-.7ex}{\ensuremath{\LieG{U}(2)}}.\]
\end{lemma}
\begin{proof}
In terms of a coframe $\eta^1,\dotsc,\eta^5$ that makes $d$ into $(0,0,0,0,12+34)$, we compute
\begin{gather*}
Z^2=\Lambda^2\Span{\eta^1,\eta^2,\eta^3,\eta^4},\\
Z^3=\Lambda^3\Span{\eta^1,\eta^2,\eta^3,\eta^4}\oplus \eta^5\wedge\Span{\eta^{12}-\eta^{34},\eta^{13},\eta^{14},\eta^{23},\eta^{24}}.
\end{gather*}
Thus $\eta^1,\dotsc, \eta^4$ and $e^1,\dotsc, e^4$ span the same space; in particular $e^1,\dotsc, e^4$ are closed and $de^5$ is in $\Lambda^2\Span{e^1,\dotsc,e^4}$. Then
\[0=d(e^5\wedge\omega_2)=de^5\wedge\omega_2,\]
and similarly $de^5\wedge\omega_3$ is zero. Then
\[de^5\in \Span{e^{12}+e^{34}}\oplus\Lambda^2_-, \quad \Lambda^2_-=\Span{e^{12}-e^{34},e^{13}-e^{42},e^{14}-e^{23}}.\]
Since $\SU(2)$ acts transitively on the sphere in $\Lambda^2_-$, we can assume that
$de^5$ is $\lambda\omega_1+\mu (e^{12}-e^{34})$, with $\mu$ non-negative. Moreover $\mu^2-\lambda^2$ cannot be zero, for otherwise $(de^5)^2$ is zero.

The last part of the statement follows from the fact that when $\mu^2-\lambda^2>0$, $\lambda\omega_1+\mu (e^{12}-e^{34})$ is in the same $\LieG{U}(2)$-orbit as
\[\lambda(e^{12}+e^{34})-\sqrt{\mu^2-\lambda^2}(e^{14}-e^{23})-\lambda(e^{12}-e^{34}).\qedhere\]
\end{proof}

\begin{lemma}
Let $d$  be in the $\GL(5,\R)$-orbit of $(0,0,0,0,12)$ in $\mathcal{D}$. Then
\[H_d = \{d_{y,k,\mu}\mid y^2=k\mu, (k,\mu)\neq(0,0)\}\!/\raisebox{-.7ex}{\ensuremath{\LieG{U}(2)}},\]
where
\[d_{y,k,\mu}=\bigl(0, 0,0, -ye^{34}-\mu e^{35},ke^{34}+y e^{35}\bigr).\]
In particular $H_d\subset\mathcal{M}_2$.
\end{lemma}
\begin{proof}
Suppose $d$ is hypo. There is a well defined filtration $V^2\subset V^4\subset(\R^5)^*$, where $\Lambda^2V^2$ is spanned by an exact two-form, and $V^4=\ker d$. Up to $\SU(2)$ action, we can assume that $e^1,e^2,e^3$ are in $V^4$.

There are two cases to consider.

\emph{a}) If $e^5$ is in $V^4$, then $e^4$ is not in $V^4$, so $d\omega_1=0$ implies that $e^3$ is in $V^2$. Closedness of $\psi_2$, $\psi_3$ gives
\[ e^{25}\wedge de^4=0= e^{15}\wedge de^4,\]
so $V^2$ is spanned by $e^3,e^5$. Thus $d=d_{0,0,\mu}$, with $\mu$ nonzero.

\emph{b}) If $e^5$ is not in $V^4$, then $e^4-ae^5$ is in $V^4$ for some $a\in\R$. Then
\begin{align*}
e^{135}+e^{425} &= e^{13}\wedge e^5 + (e^4-ae^5)\wedge e^{25},\\
e^{145}+e^{235} &= e^{23}\wedge e^5 - (e^4-ae^5)\wedge e^{15}.
\end{align*}
Thus
\begin{equation}
 \label{eqn:de5wedge}
de^5\wedge (e^{13}+(e^4-ae^5)\wedge e^2)=0 = de^5\wedge (e^{23}+e^1\wedge(e^4-ae^5)).
\end{equation}
We can also assume that $e^3$ is in $V^2$. Indeed if $e^4$ is in $V^4$, it suffices to act by an element of $\SU(2)$ to obtain $e^3\in V^2$. If on the other hand $e^4$ is not in $V^4$, then $d\omega_1=0$ implies $e^3\in V^2$. It follows that $de^5$ is a linear combination of $e^{13},e^{34}-ae^{35},e^{23}$. By \eqref{eqn:de5wedge},
\[de^5=k(e^{34}-ae^{35}),\quad  de^4=ak(e^{34}-ae^{35}),\]
so $d$ has the form $d_{y,k,\mu}$.
\end{proof}

\begin{lemma}
If $d$ is in the $\GL(5,\R)$-orbit of $(0,0,0,12,13)$, then
\[H_d = \{d_{k,\mu}\mid k\geq\mu, k,\mu\neq0\}\!/\raisebox{-.7ex}{\ensuremath{\LieG{U}(2)}},\]
where
\[d_{k,\mu}=(0,k e^{15}, 0, \mu e^{35},0).\]
In particular $H_d\subset\mathcal{M}_1$.
\end{lemma}
\begin{proof}
In terms of a coframe $\eta^1,\dotsc,\eta^5$ which makes $d$ into $(0,0,0,12,13)$, we compute
\begin{align*}
Z^2&=\Span{\eta^{12},\eta^{13},\eta^{14},\eta^{15},\eta^{23},\eta^{24},\eta^{25}+\eta^{34},\eta^{35}},\\
Z^3&=\Span{\eta^{123},\eta^{124},\eta^{125},\eta^{134},\eta^{135},\eta^{145},\eta^{234},\eta^{235}}.
\end{align*}
Suppose $d$ is hypo. Recall from \cite{ContiFernandezSantisteban} that if $X\in\R^5$, $\beta\in(\R^5)^*$ are such that
\[\dim(X\hook Z^3)\wedge\beta<2,\]
then $e^5(X)$ is zero. In this case, if $\eta_1,\dotsc, \eta_5$ is the frame dual to $\eta^1,\dotsc, \eta^5$,
\[\dim (\eta_4\hook Z^3)\wedge \eta^1=1=\dim (\eta_5\hook Z^3)\wedge \eta^1;\]
so $e^5$ is a linear combination of $\eta^1,\eta^2,\eta^3$.

Suppose $e^5$ is linearly independent of $\eta^1$. Then
\[\dim \eta^1\wedge e^5\wedge \Span{\omega_1, \omega_2, \omega_3}=3;\]
on the other hand, if $e^5=a\eta^2+b\eta^3 \pmod{\eta^1}$,
\[\eta^1\wedge e^5\wedge Z^2 \subset\Span{\eta^{1234},\eta^{1235}}=\eta^1\wedge Z^3,\]
which is absurd. So $e^5$ is linearly dependent of $\eta^1$.

Let $V^3=\ker d$. Then $\Span{e^1,e^2,e^3,e^4}$ intersects $V^3$ in a two-dimensional space; up to $\SU(2)$ action, we can assume $V^3$ contains $e^1$. Suppose that $\Span{e^3,e^4}$ has trivial intersection with $V^3$. Then $de^3$ and $de^4$ are non-zero, contradicting
\[de^3\wedge e^4-e^3\wedge de^4=de^{34}=-de^{12}\in\Lambda^3V^3.\]
Thus, $\Span{e^3,e^4}$ intersects $V^3$ in a one-dimensional space; up to $\LieG{U}(2)$ action, we can assume that $e^3$ is in $V^3$. Therefore
\[de^2 = ae^{15}+be^{35}, \quad de^4=he^{15}+ke^{35};\]
using the fact that $\omega_1$ is closed, we see that $h=b$ and obtain
\[\begin{pmatrix}de^2\\ de^4\end{pmatrix}= \begin{pmatrix} a &  b\\ b & k\end{pmatrix}\begin{pmatrix}e^{15}\\ e^{35}\end{pmatrix}.\]
This matrix can be assumed to be diagonal, because the symmetry group $\LieG{U}(2)$ contains an $S^1$ that rotates $\Span{e^1,e^3}$ and $\Span{e^2,e^4}$, whose action corresponds to conjugating the matrix by $\SO(2)$. So we obtain $d_{k,\mu}$ as appears in the statement, where $k$ and $\mu$ can be interchanged by the same circle action.
\end{proof}

Theorem~\ref{thm:hierarchy} now follows trivially from the lemmas.
\section{Integrating the flow}
\label{sec:integrating}
In this section we compute explicitly the hypo evolution flow in $\mathcal{D}$ for each nilpotent hypo Lie algebra, giving a description of the resulting orbits. In particular, we conclude that the orbits are semi-algebraic sets in the affine space $\mathcal{D}$ (i.e. they are defined by polynomial equalities and inequalities), each orbit has at most one limit point, and there are no periodic orbits.

Moreover the only critical point is the origin of $\mathcal{D}$: this means that the evolution flow in $(D)$ is always transverse to the group of automorphisms, except in the abelian case, when the flow is constant.

In fact, the vector fields on the families $\mathcal{M}_i$ that generate the evolution flow are homogeneous of degree $2$ in the coordinates of $\mathcal{D}$; this is a consequence of the construction, and implies that multiplication by a scalar in $\mathcal{D}$ maps orbits into orbits. So, we could also think of the flow as taking place in $\mathbb{P}(\mathcal{D})$. However, we shall refrain from using the projective quotient as it does not appear to simplify the description of the orbits.

Recall that there are three families of nilpotent hypo Lie algebras, which have to be studied separately. We begin with $\mathcal{M}_1$, by computing first integrals for the flow.
\begin{lemma}
\label{lemma:invariantcurves}
The evolution in $\mathcal{M}_1$ of
\[d_{\lambda,\mu,h,k}=(\lambda e^{35},he^{35}+ke^{15}, 0,(-\lambda e^2+he^1+\mu e^3)\wedge e^5,0)\]
leaves the following sets invariant:
\begin{align*}
O_1&=\{k=0=h=\lambda\},\\
O_2^{A}&=\left\{h=0, \lambda =0,  \frac{(\mu-k)^4}{\mu^3k^3}=A \right\},\\
O_3^{AP}&=\left\{ k=0, \frac{(\mu^2+4(h^2+\lambda^2))^2}{(h^2+\lambda^2)^3}= A,\quad [h:\lambda]\equiv P\in\RP^1\right\},\\
O_4^{AB}&=\left\{ h=0, \frac{(\lambda^2+3k^2+Bk^2\lambda)^3}{k^4\lambda^4}=A, \mu=\frac1{3k}(-2\lambda^2 +3k^2+Bk^2\lambda)\right\},\\
O_5^{AB}&=\left\{\lambda=0, \mu = \frac{Ah^4}{k^5}+\frac{h^2}k, Bh^3k^2+h^2k^4-Ah^4 +  k^4=0\right\},\\
O_6^{ABC}&=\left\{B\lambda(3h^2+\lambda^2+Ah\lambda)^2-3h^3(C\lambda-B)^2=0,\mu =\frac{ h^2-\lambda^2-\frac13Ah\lambda}k, k^4=\frac{Bh^3}{3\lambda}\right\}.
\end{align*}
Moreover every periodic orbit is contained in some $O_2^A$ or $O_4^{AB}$.
\end{lemma}
\begin{proof}
We compute
\begin{align*}
d\omega_2&= (he^1+\mu e^3)\wedge e^{52}-e^4\wedge(he^{35}+ke^{15})=e^5\wedge(-he^{12}+\mu e^{23}+he^{34}+ke^{14}),\\
d\omega_3&=\lambda e^{354}  -e^1\wedge (-\lambda e^2+\mu e^3)\wedge e^5 -ke^{135}=e^5\wedge(-\lambda e^{34}+\lambda e^{12} -\mu e^{13}-ke^{13}).
\end{align*}
Therefore the components of the intrinsic torsion can be given by
\begin{gather*}
\beta=0, f=0, \omega^-=0, (\sigma_2^-)_a=-h/2, (\sigma_2^-)_c=\frac{k-\mu}4, g=\frac{k+\mu}2,\\
(\sigma_3^-)_a=\lambda/2, (\sigma_3^-)_b=-\frac{\mu+k}4.
\end{gather*}
The induced infinitesimal gauge transformation determined by the hypo evolution flow is given by
\[ \hat X_{d_{\lambda,\mu,h,k}}=\left(\begin{array}{ccccc}\frac{1}{2}  k&0&\frac{1}{2}  h&-\frac{1}{2}  \lambda&0\\0&-\frac{1}{2}  k&-\frac{1}{2}  \lambda&-\frac{1}{2}  h&0\\\frac{1}{2}  h&-\frac{1}{2}  \lambda&\frac{1}{2}  \mu&0&0\\-\frac{1}{2}  \lambda&-\frac{1}{2}  h&0&-\frac{1}{2}  \mu&0\\0&0&0&0&\frac{1}{2} k+\frac{1}{2} \mu\end{array}\right).\]
Thus
\begin{multline*}
\tilde X_{d_{\lambda,\mu,h,k}}=\mu_{*e}(\hat X(d_{\lambda,\mu,h,k}))d_{\lambda,\mu,h,k}=\biggl(-\frac{3}{2}  \lambda \mu e^{35}+ \lambda^{2} e^{25}- \lambda e^{15} h,\\
-\frac{3}{2}  k^{2} e^{15}+\frac{1}{2}  k \lambda e^{45}-\frac{3}{2}  \mu e^{35} h-\frac{3}{2}  k e^{35} h-\frac{1}{2}  k \mu e^{15}+ \lambda e^{25} h- e^{15} h^{2},
-\frac{1}{2}  k \lambda e^{15},\\
\frac{3}{2}  \lambda \mu e^{25}- \lambda^{2} e^{35}-\frac{3}{2}  \mu e^{15} h-\frac{3}{2}  k e^{15} h-\frac{1}{2}  k \mu e^{35}-\frac{3}{2}  \mu^{2} e^{35}- e^{35} h^{2},0\biggr).
\end{multline*}
This vector field is not tangent to $\mathcal{M}_1$. However, we can project to a vector that is tangent to the family by using the $\SU(2)$-invariance: indeed, adding an appropriate element in
$\mu_{*e}(\su(2))d_{\lambda,\mu,h,k}$,
 we obtain
\begin{multline*}\biggl(- \lambda \mu e^{35},-\frac{1}{2}  {(3 k^{2}+ k \mu)} e^{15}- {(2  k h+ \mu h)} e^{35},0,\\
-\frac{1}{2}  {(4 \lambda^{2}+ k \mu+3 \mu^{2}+4 h^{2})} e^{35}+ \lambda \mu e^{25}- {(2  k h+ \mu h)} e^{15},0\biggr).
\end{multline*}
We conclude that evolution in $\mathcal{M}_1$ is given by
\[\begin{cases}\lambda'=-\mu\lambda\\ h'=-\mu h -2hk\\ k'=-\frac{1}{2}  \mu k -\frac{3}{2}  k^{2} \\ \mu'=-\frac{3}{2}  \mu^{2} -2  h^{2} -2  \lambda^{2} -\frac{1}{2}  \mu k \end{cases}\]

The only critical point is the origin. In order to describe the other orbits, we make use of certain first integrals that were found by trial and error using a computer. A posteriori, it is straightforward to verify that the following rational functions are first integrals on their domain of existence:
\[\frac{h^2-\lambda^2-k\mu}{h\lambda}, \quad \frac{k^4\lambda}{h^3}, \quad\frac{k^2(6h^2  -2\lambda^2-3\mu k + 3k^2)}{h^3}.\]
The invariant set $O_6^{ABC}$ is obtained by assuming $k\neq0$ and setting these functions equal to $A/3$, $B/3$ and $C$ respectively.

More than these three first integrals, what is relevant is the field extension of $\R$ they generate. This field also contains rational functions that are defined on the invariant hyperplane $h=0$. Indeed, an alternative description of the same curves, also valid for $h=0$, can be obtained by setting
\[A=\frac{h^2-\lambda^2-k\mu}{(k\lambda)^{4/3}}, \quad B= \frac{h^3}{k^4\lambda}, \quad C=\frac{6h^2  -2\lambda^2-3\mu k + 3k^2  }{k^2\lambda}.\]
Then
\begin{align*}
\mu k&=-A(k\lambda)^{4/3}+(Bk^4\lambda)^{2/3}-\lambda^2,\\
Ck^2\lambda&=3(Bk^4\lambda)^{2/3}  +\lambda^2+3A(k\lambda)^{4/3} + 3k^2 .
\end{align*}
For $B=0$, up to renaming the constants, we obtain the family denoted in the statement by $O_4^{AB}$.

Similarly, we can describe the orbits contained in $\{\lambda=0\}$ by the two first integrals
\[
 -A=\frac{k^4(h^2-\mu k)}{h^4}, \quad -B=\frac{k^2(2h^2-\mu k +  k^2)}{h^3}.\]
This description gives rise to the curve
\[\mu = \frac{Ah^4}{k^5}+\frac{h^2}k, \quad Bh^3k^2+h^2k^4-Ah^4 +  k^4=0,\]
denoted in the statement by $O_5^{AB}$.

In the invariant hyperplane $\{k=0\}$, the evolution equations imply $h'/h=\lambda'/\lambda$. Assuming $h,\lambda$ are not both zero, it follows that  the point $[h:\lambda]$ of $\RP^1$ does not vary along orbits; equivalently, we can set
\[h=(\cos B) s,\quad \lambda=(\sin B) s, \quad s=\sqrt{h^2+\lambda^2},\]
where $B$ is a constant. Assuming both $\mu$ and $s$ are nonzero, we can use $s$ as a new variable, and compute explicitly
\[\mu^2=-4s^2+Ks^3,\]
showing that $(\mu^2+4s^2)/s^3$ is a first integral. This is also true when $\mu=0$, so $O_3$ is indeed invariant.

Restricting now to the invariant plane $\{h=0,\lambda=0\}$, the evolution equations become
\[\begin{cases} k'/k=-\frac{1}{2}  \mu  -\frac{3}{2}  k \\ \mu'/\mu=-\frac{3}{2}  \mu -\frac{1}{2}   k \end{cases}\]
Up to a change of variable, we obtain
\[\frac{(k^3\mu^{-1})'}{k^3\mu^{-1}}=-k/\mu, \quad \frac{(k\mu^{-3})'}{k\mu^{-3}}=1,\]
which can be solved explicitly, showing that
$(\mu-k)^4/(k\mu)^3$ is a first integral. This proves that $O_2^A$ is invariant.

Finally, the invariance of $O_1$ is obvious.

\smallskip
The fact that all periodic orbits are contained in the families $O_2^A$, $O_4^{AB}$ can be deduced directly from the evolution equations by producing a quantity that varies monotonously along the flow. Indeed, if $kh\neq0$, then $k^2/h$ is monotonous. If $k=0$, then $\lambda^2+h^2$ is either monotonous or identically zero, in which case we are in $O_1$, that obviously contains no periodic orbit. If $k\neq0$ and $h=0$, we are in the families $O_2^A$, $O_4^{AB}$.
\end{proof}

Having computed invariant curves for the flow, we can now give a more explicit description of the orbits. This amount essentially to detecting the connected components of the invariant curves appearing in Lemma~\ref{lemma:invariantcurves}.

\begin{theorem}
\label{thm:evolveM1}
With respect to the hypo evolution of $\mathcal{M}_1$, all orbits are semi-algebraic subsets of $\mathcal{D}$ and have at most one limit point, namely the unique critical point $d_{0,0,0,0}=0$; in particular, there is no periodic orbit. Orbits are mapped into orbits by the transformations sending $d_{\lambda,\mu,h,k}$ into
\[d_{-\lambda,\mu,h,k}, \quad d_{\lambda,\mu,-h,k}, \quad d_{\lambda,-\mu,h,-k}.\]
Up to these symmetries, the non-trivial orbits are
\begin{gather*}
\{h=0=k=\lambda, \mu>0\};  \quad \{\mu=k>0, h=0=\lambda\};\quad
O_2^A\cap \{\mu>k, \pm k>0\}, A>0;\\
O_2^A\cap \{\mu>k\}, A<0;\quad
O_3^{AP}\cap \{h,\lambda\geq0\}, A>0, P\in\RP^1;\\O_4^{AB}\cap\{k>0,\lambda>0\};\quad
O_5^{AB}\cap\{h>0,k>0\}, \quad (A,B)\neq(0,0);
\end{gather*}
and the connected components of $O_6^{ABC}$, $B,C>0$. The latter can be parametrized as
\[\lambda(s)=\frac{s^{3}B^2}{s^4B^{2}+3+ABs^2+BCs^{3}}, \quad h(s)=\frac13s^2B\lambda, \]
\[k(s)=\frac{B}{3}\sqrt{\abs{s^{3}\lambda}}, \quad \mu(s) =\frac{ h^2-\lambda^2-\frac13Ah\lambda}k\]
where $s$ ranges in any maximal interval in $\R^*$ where the denominator of $\lambda(s)$ is nonzero.
\end{theorem}

\begin{proof}
The only critical point is the origin, so all orbits are connected, smooth immersed curves. Thus, we only need determine the connected components of the invariant curves $O_i$. We have six cases to consider.

\emph{i}) The invariant set $O_1$ is a line containing the origin, so it contains three orbits, all non-periodic, each with the origin itself as the unique limit point. The half-line $\mu<0$ is obtained from the half-line $\mu>0$ by a symmetry.

\emph{ii}) In order to study $O_2^{A}$, introduce the variable $s=\mu-k$. If $A=0$ then $s\equiv0$; so there is no periodic orbit and the origin is the only limit point. Assume that $A$ is nonzero; then $s\neq0$ on $O_2^A$, so $s>0$ and $s<0$ define distinct components. Up to a  symmetry, we can assume $s>0$. For each fixed $s$, we get the two points
\[\mu=k+s, \quad k=\frac12\left(-s\pm \sqrt{s^2+4A^{-1/3}s^{4/3}}\right).\]
So if $A>0$, to each value of $s$ correspond two points, one with $k>0$ and the other with $k<0$. Therefore, $O_2^A$ contains two distinct orbits, both non-periodic with a single limit point, namely the origin.
On the other hand if $A<0$, there are two points over $s$ if $s^2>-64/A$, and one point if equality holds. So in this case there is a single orbit, non-periodic, with no limit points.

\emph{iii}) We can rewrite $O_3^{AP}$ using the equation
\[\mu^2=\sqrt{A}r^3-4r^2, \quad r^2=h^2+\lambda^2;\]
this shows immediately that $A$ has to be positive. Using the symmetries, we can restrict to $\{h,\lambda\geq0\}$, which intersects $O_3^A$ in a single orbit, non-periodic, with no limit points.

\emph{iv})
Now consider $\tilde O_4^{AB}=O_4^{AB}\cap\{k>0,\lambda>0\}$, defined by the equation
\[\frac{(\lambda^2+3k^2+Bk^2\lambda)^3}{k^4\lambda^4}=A,\]
that we rewrite as
\[k^2(3+B\lambda)-k^{4/3}(A^{1/3}\lambda^{4/3})+\lambda^2 =0.\]
For fixed values of $A,B,\lambda$, by setting $t=k^{2/3}$ we obtain a  polynomial in $t$, namely
\[p_{AB\lambda}(t)=t^3(3+B\lambda)-t^{2}A^{1/3}\lambda^{4/3}+\lambda^2\]
which for short we rewrite as
\[p(t)=at^3+bt^2+c,\]
where $c>0$. Assuming for the moment that $a,b\neq0$, the Sturm sequence of $p$ is
\[p, 3at^2+2bt, \frac29 \frac{b^2}{a}t-c, -\frac{9ac}b - \frac{243a^3c^2}{4b^4}.\]
The number of positive roots can be computed using the Sylvester theorem (see e.g. \cite{RealAlgebraicGeometry}). To that end, we compute the number $\nu(p,t)$ of sign changes of the Sturm sequence evaluated at $t$; remembering that $c>0$, we find
\begin{multline*}
\nu(p,+\infty)=\operatorname{signchanges}\left(a, 3a, \frac29 \frac{b^2}{a} ,  -\frac{9ac}b - \frac{243a^3c^2}{4b^4}\right)\\
=\operatorname{signchanges}\left(1,1,1,-4b^3-27a^2c\right)
 \end{multline*}
and
\begin{multline*}
\nu(p,0)=\operatorname{signchanges}\left(c, 0, -c,   -\frac{9ac}b - \frac{243a^3c^2}{4b^4}\right)\\
=\operatorname{signchanges}\left(1, 0, -1,  a(- 4b^3-27a^2c)\right) 
\end{multline*}
The number of positive roots of $p$ equals $\nu(p,0)-\nu(p,+\infty)$, by the Sylvester theorem; when either $a$ or $b$ are zero, it can be determined directly. So there are the following possibilities:
\begin{align*}
 a&>0, b\neq0,-4b^3-27a^2c>0& &\implies \text{ two positive roots,}\\
 a&<0,\quad \text{or } a=0,b<0& &\implies \text{ one positive root,}\\
a&>0,b\neq0,-4b^3-27a^2c=0 & &\implies \text{ one positive double root,}
\end{align*}
and no positive root otherwise.
Returning to the original variables, we have
\[-4b^3-27a^2c\geq0 \iff
\frac{(3+B\lambda)^2}{\lambda^2}\leq\frac{4}{27}A.\]
We claim that $\tilde O_4^{AB}$ is either empty or connected. Indeed, for fixed $A,B$ let
\[S_k=\{\lambda>0\mid p_{AB\lambda}(t) \text{ has at least $k$ positive roots}\}.\]
Then $\tilde O_4^{AB}$ has $2$ points over each point of $S_2$. If $S_2$ is not empty, it is easy to check that it is connected and its boundary contains exactly one $\lambda_b\in\R$ corresponding to a double root. Hence, the part of the curve lying over $S_2\cup\{\lambda_b\}$ is connected. Likewise, $S_1$ is connected. It follows that $\tilde O_4^{AB}$ is either empty, or diffeomorphic to an interval. So, it consists of a single non-periodic orbit.

By construction, the limit set is necessarily contained in some hyperplane $\lambda\equiv\lambda_0$. If $\lambda_0\neq0$, this hyperplane intersects each invariant set appearing in Lemma~\ref{lemma:invariantcurves} in a discrete set and contains no critical point. It follows that the limit set is contained in $\lambda\equiv0$. On the other hand, inside $\tilde O_4$, if $\lambda$ goes to zero then so do $k$ and $\lambda^2/k$. Since
\[\mu=\frac1{3k}(-2\lambda^2 +3k^2+Bk^2\lambda),\]
it follows that the only limit point is the origin.

\emph{v})
Now consider the curve $\tilde O_5^{AB}=O_5^{AB}\cap\{h>0,k>0\}$, defined by \[Bh^3k^2+h^2k^4-Ah^4 +  k^4=0.\]
We assume that $A,B$ are not both zero, for otherwise $k=0$. Solving in $t=k^2$, we find
\[t=\frac{-Bh^3\pm h^2\sqrt{(B^2+4A) h^2+4A}}{2(h^2+1)}.\]
\begin{itemize}
 \item If $4A\leq -B^2$, there is no positive solution.
 \item If $-B^2<4A\leq0$ and $B>0$, there is no positive solution.
 \item If $-B^2<4A<0$ and $B<0$, there are two positive solutions for $h^2$ greater than $-4A/(B^2+4A)$, and one when equality holds.
  \item If $A=0$ and $B<0$, there is one positive solution for all $h>0$.
 \item If $A>0$, there is one positive solution for all $h>0$.
\end{itemize}
We conclude that $\tilde O_5^{AB}$ is either empty or connected.

As for the limit set, it is contained in $h=0$. By definition of $\tilde O_5^{AB}$, $k$, $h^2/k$ go to zero when $h$ goes to zero. Hence, if $A=0$,  $\mu=h^2/k$ also goes to zero, and the only limit point is the origin. On the other hand if $A>0$, then $k^2/h^2$ goes to $\sqrt{A}$ when $h$ goes to zero, so $\mu$ goes to infinity and there is no limit point.

\emph{vi}) We now turn to $O_6^{ABC}$, whose defining equation is
\[B\lambda(3h^2+\lambda^2+Ah\lambda)^2-3h^3(C\lambda-B)^2=0.\]
Two of the symmetries in the statement have the effect of changing the signs of $C,\lambda,h$ and $A,B,\lambda$ respectively. Thus we can assume $C>0$ and $B>0$.

Introduce a variable $z=\frac{h}{B\lambda}$; then $z>0$ on $O_6^{ABC}$.
We compute
\[\lambda^2(3z^2B^{2}+1+ABz)^2-3B^2z^3(C\lambda-B)^2=0,\]
or equivalently
\[\lambda(3z^2B^{2}+1+ABz)\pm\sqrt{3} Bz^{3/2}(C\lambda-B)=0.\]
Setting $s=\pm\sqrt {3z}$, we obtain
\[\lambda=\frac{s^{3}B^2}{s^4B^{2}+3+ABs^2+BCs^{3}}.\]
By construction, $h=\frac13s^2B\lambda$; moreover, by definition of $O_6^{ABC}$ we have
\[k^4=\frac1{81}B^4s^6\lambda^2.\]
Up to a symmetry, we can impose that $k>0$, and obtain the curve appearing in the statement. It is now obvious that only the origin can be a limit point.
\end{proof}

This concludes our study of the first family. Recall that the second family $\mathcal{M}_2$ consists of elements of the form
\begin{multline*}
d_{x,y,h,k,\lambda,\mu}=
\bigl(0, xe^{34}+\lambda e^{35},0, x(e^{14}-e^{23})-ye^{34}+\lambda e^{15}-\mu e^{35},\\
-h(e^{14}-e^{23})+ke^{34}-x e^{15}+y e^{35}\bigr).
\end{multline*}
More precisely, $\mathcal{M}_2$ is an algebraic variety in $\mathcal{D}$ which can be described as the union
\[\mathcal{M}_2=\bigcup_{\substack{l_0^2 =l_1l_2}} \mathcal{M}_{2,l}, \quad \mathcal{M}_{2,l}=\{d_{x,y,h,k,\lambda,\mu}\mid (x,h,\lambda), (y,k,\mu)\in l\},\]
where $l=[l_0:l_1:l_2]$ is viewed as both a point in $\RP^2$ and a line in $\R^3$.

\begin{theorem}
\label{thm:evolveM2}
With respect to the hypo evolution of $\mathcal{M}_2$, all orbits are semi-algebraic subsets of $\mathcal{M}_2$ and have at most one limit point, namely the unique critical point $d_{0,\dots,0}=0$; in particular, there is no periodic orbit. Orbits are mapped into orbits by the transformations sending $d_{x,y,h,k,\lambda,\mu}$ into
\[
d_{-x,y,-h,k,-\lambda,\mu}, \quad d_{x,-y,h,-k,\lambda,-\mu}, \quad d_{-x,-y,h,k,\lambda,\mu}.
\]
Up to these symmetries, the non-trivial orbits are
\begin{gather*}
O_{Al}=\left\{d_{x,y,h,k,\lambda,\mu}\in\mathcal{M}_{2,l}\mid (h+\lambda)^3=A((\mu+k)^2+4(h+\lambda)^2),  x, y,h,\lambda,k,\mu>0\right\} , A>0; \\
O_l=\left\{d_{x,y,h,k,\lambda,\mu}\in\mathcal{M}_{2,l}\mid x=0=h=\lambda, y,k,\mu\geq 0, (k,\mu)\neq(0,0)\right\};
\end{gather*}
where $l$ varies among points of the quadric $l_0^2=l_1l_2$.
\end{theorem}
\begin{proof}
Proceeding as in Lemma~\ref{lemma:invariantcurves}, we compute
\[\hat X_{d_{x,y,h,k,\lambda,\mu}}=\left(\begin{array}{ccccc}0&0&\frac{1}{2}  \lambda+\frac{1}{2}h&0&0\\0&0&0&-\frac{1}{2}  \lambda+\frac{1}{2} h&x\\\frac{1}{2}  \lambda+\frac{1}{2} h&0&-\frac{1}{2} k-\frac{1}{2}  \mu&0&0\\0&-\frac{1}{2}  \lambda+\frac{1}{2} h&0&-\frac{1}{2} k+\frac{1}{2} \mu&- y\\0&x&0&- y&\frac{1}{2} k-\frac{1}{2}  \mu \end{array}\right).\]
The resulting ODE is
\begin{gather*}
\begin{cases}x'&=x(\mu+k)\\
\lambda'&=\lambda(\mu+k)\\
h'&=h(\mu+k)\\
\mu'&= 2\lambda(h+\lambda)+\frac32\mu(\mu+k)\\
y' &= 2x(h+\lambda)+\frac32y(\mu+k)\\
k'&=2h(h+\lambda)+\frac32k(\mu+k)\\
\end{cases}
\end{gather*}
Critical points are given by the subspace $\{\mu+k=0, h+\lambda=0\}$, which intersects $\mathcal{M}_2$ only in the origin. It is also easy to verify that each $\mathcal{M}_{2,l}$ is invariant under the flow, and that the symmetries appearing in the statement map orbits into orbits. By construction $\mu k$ and $h \lambda$ are non-negative; using $[x:h:\lambda]=[y:k:\mu]$,  we can assume up to symmetry that all parameters are non-negative.

In order to determine the integral lines, consider the associated two-dimensional system
\[\begin{cases}
 (h+\lambda)'=(h+\lambda)(\mu+k)\\
 (\mu+k)'= 2(h+\lambda)^2+\frac32(\mu+k)^2\end{cases}
\]
which has only the origin as a critical point. Outside of the origin, a first integral is given by
\[\frac{(h+\lambda)^3}{(\mu+k)^2+4(h+\lambda)^2}\equiv A.\]
Since $h+\lambda\geq0$, necessarily $A\geq0$. If $A=0$, we find the orbit $O_l$.
Otherwise, $A>0$ and the orbit is contained in $O_{Al}$; writing \[\mu+k=(h+\lambda)\sqrt{\frac{h+\lambda}A-4},\]
we see that $O_{Al}$ is connected, so it coincides with the orbit. It is clear that the only limit point is the origin.
\end{proof}
We now come to the last family $\mathcal{M}_3$, whose general element has the form
\[d_{\lambda,\mu}=\left(0,0,0,0,(\lambda+\mu)e^{12}+(\lambda-\mu) e^{34}\right). \]
\begin{theorem}
\label{thm:evolveM3}
With respect to the hypo evolution of $\mathcal{M}_3$, all orbits are semi-algebraic subsets of $\mathcal{M}_3$ and have at most one limit point, namely the unique critical point $d_{0,0}=0$; in particular, there is no periodic orbit. Orbits are mapped into orbits by the transformations sending $d_{\lambda,\mu}$ into
\[
d_{-\lambda,\mu}, \quad d_{\lambda,-\mu}.
\]
Up to these symmetries, the non-trivial orbits are
\begin{gather*}
\{\mu=0,\lambda>0\};\quad \left\{(\lambda^2-\mu^2)^3= A \mu^{4}, \lambda,\mu>0\right\}, A\geq0;\\
\left\{(\lambda^2-\mu^2)^3= A \mu^{4}, \mu>0\right\}, A<0.
\end{gather*}
\end{theorem}
\begin{proof}
At $d_{\lambda,\mu}$, the intrinsic torsion is
\[\beta=0, f=\lambda, g=0, \omega^-=\mu(e^{12}-e^{34}), \sigma^-_k=0\]
So
\[
 \hat X_d=\diag\left(
-\frac12(\mu+\lambda), -\frac12(\mu+\lambda) , \frac12(\mu-\lambda),\frac12(\mu-\lambda),\lambda\right).
\]
Hence
\[
\tilde X_{d_{\lambda,\mu}}= \left(0,0,0,0, (\mu^2+2\lambda^2)(e^{12}+e^{34})+3\lambda\mu(e^{12} -e^{34})\right).
\]
The resulting equations are
\[\lambda'=\mu^2+2\lambda^2,\quad \mu'=3\lambda\mu.\]
It is easy to see that the only critical point is the origin. Moreover the invariant set $\{\mu=0\}$ contains two other orbits, namely the half-lines $\pm\lambda>0$.

When $\mu^2-\lambda^2>0$, we can reduce to the case of $\mathcal{M}_2$ by
\[k=2\lambda, h=\sqrt{\mu^2-\lambda^2};\]
by Theorem~\ref{thm:evolveM2}, $h^3/(k^2+4h^2)$ is a first integral, which in terms of $\lambda$ and $\mu$ gives
\[(\lambda^2-\mu^2)^3\mu^{-4};\]
one can check directly that this is a first integral on all of $\mu\neq0$.
Up to the symmetries, we can assume that $\mu>0$.
Writing the curve as
\[\lambda^2=\mu^{4/3}(A+\mu^{2/3}),\]
we see that  it has two connected components in the half-plane $\mu>0$ if $A\geq0$, and one otherwise. In the first case, we can use the symmetries and assume $\lambda>0$.

The facts that all orbits are non-periodic and that the only limit point is the origin are now obvious.
\end{proof}

\begin{remark}
The non-existence of periodic orbits can also be seen as a consequence of the Cheeger-Gromoll splitting theorem (see \cite{Besse}). Indeed a periodic orbit gives a solution to the evolution equations defined on all of $\R$, meaning that the $6$-manifold is complete and contains a line; by the Cheeger-Gromoll splitting theorem, this means that it is a product. Then the Weingarten tensor $-Q_{P_t}$ is zero: this implies that the orbit consists of a single point.
\end{remark}
\begin{remark}
Multiplication by a scalar in $\mathcal{D}$ amounts to a change in time scale in terms of the evolution flow, and an overall scale change in terms of the resulting six-dimensional metric. One could therefore eliminate one parameter and consider these metrics up to rescaling. The disadvantage of doing so would be losing track of the relations between the orbits.
\end{remark}
\section{Examples}
\label{sec:examples}
In this section we  show with examples that the classification of orbits in terms of $\mathcal{D}$ is sufficient to determine the geometric properties of the corresponding $6$-dimensional metrics.

\begin{example}
In this first example we consider the orbits $O_l$ in $\mathcal{M}_2$, and show how the corresponding one-parameter family of $\SU(2)$-structures on the Lie algebra $(0,0,0,0,12)$ can be recovered. 

The generic element of $O_l$ has the form
\[d_{y,k,\mu}=\bigl(0, 0,0, -ye^{34}-\mu e^{35},ke^{34}+y e^{35}\bigr),\quad y^2=k\mu.\]
We parametrize $O_l$  by setting
\[\mu=s\cos^2\theta, \quad k=s\sin^2\theta, \quad y=s\sin\theta\cos\theta,\]
where $s=\mu+k$ and $\theta$ is in $[0,\pi/2)$. Then
\[d_{s,\theta}=\bigl(0, 0,0, -s\sin\theta\cos\theta e^{34}-s\cos^2\theta e^{35},s\sin^2\theta e^{34}+s\sin\theta\cos\theta e^{35}\bigr).\]
From the evolution equations $s'=\frac32s^2$; up to time translation, this gives
\[s=\frac2{1-3t}.\]
By the calculations in the proof of Theorem~\ref{thm:evolveM2},
\[\hat X_{d(t)}=\frac{1}{1-3t}\left(\begin{array}{ccccc}0&0&0&0&0\\0&0&0&0&0\\0&0&-1&0&0\\0&0&0&\cos 2\theta&-\sin 2\theta\\0&0&0&-\sin2\theta &-\cos2\theta \end{array}\right).\]
The eigenspaces of this matrix do not depend on $t$; in particular, setting
\[E(t)=\cos\theta \,e^4(t) -\sin\theta\, e^5(t), \quad F(t)=\sin\theta\, e^4(t)+\cos\theta\, e^5(t),\]
we find
\[E'(t)=\frac{1}{1-3t}E(t), \quad F'(t)=\frac{-1}{1-3t}F(t),\]
hence
\[E(t)=(1-3t)^{-1/3}E(0), \quad F(t)=(1-3t)^{1/3}F(0).\]
Similarly, we compute
\[e^1(t)=e^1(0), \quad e^2(t)=e^2(0), \quad e^3(t)=(1-3t)^{1/3}e^3(0).\]
Now observe that $F(0)$ is closed and $dE(0)=-2e^3(0)\wedge F(0)$. Thus, up to an automorphism we can assume
\[\eta^1=F(0), \quad \eta^2=2e^3(0), \quad \eta^3=e^1(0), \quad \eta^4=e^2(0), \quad \eta^5=E(0).\]
Then
\begin{gather*}
e^1(t)=\eta^3, \quad e^2(t)=\eta^4, \quad e^3(t)=\frac12(1-3t)^{1/3}\eta^2 \\
e^4(t)=\cos\theta(1-3t)^{-1/3}\eta^5 +\sin\theta (1-3t)^{1/3}\eta^1\\
e^5(t)=-\sin\theta(1-3t)^{-1/3}\eta^5+\cos\theta(1-3t)^{1/3}\eta^1
\end{gather*}
In terms of the defining forms $(\alpha,\omega_i)$, the corresponding one-parameter family of $\SU(2)$-structures is given by
\begin{gather*}
\alpha(t)= -\sin\theta(1-3t)^{-1/3}\eta^5+\cos\theta(1-3t)^{1/3}\eta^1\\
\omega_1(t)=\eta^{34}+\frac12\cos\theta\eta^{25} -\frac12\sin\theta (1-3t)^{2/3}\eta^{12}\\
\omega_2(t)=-\frac12(1-3t)^{1/3}\eta^{23} -\cos\theta(1-3t)^{-1/3}\eta^{45} +\sin\theta (1-3t)^{1/3}\eta^{14}\\
\omega_3(t)=\cos\theta(1-3t)^{-1/3}\eta^{35} -\sin\theta (1-3t)^{1/3}\eta^{13}-\frac12(1-3t)^{1/3}\eta^{24}
\end{gather*}
One can easily verify that the evolution equations are indeed satisfied.
\begin{remark}
In this special case the infinitesimal gauge transformation has constant eigenspaces, which is why the evolution has a ``diagonal'' form. The other orbits with this property are those contained in $O_1$, $O_2^A$ (whose corresponding metrics are studied in \cite{ContiSalamon}) and $\mathcal{M}_3$.
 \end{remark}
\end{example}

\begin{example}
This example shows that it is not necessary to integrate the infinitesimal gauge transformation in order to determine whether the $6$-dimensional metric corresponding to an orbit is reducible or its holonomy equals $\SU(3)$.

Given a solution of the hypo evolution equations on a Lie algebra $\lie{g}$, let $g_t$ be the underlying one-parameter family of metrics on $\lie{g}$; let $\Omega_t\in\Lambda^2\lie{g}^*\otimes\so(5)$ be the corresponding one-parameter family of curvature forms. Let $g$ be the corresponding generalized cylinder metric on $G\times(a,b)$; by the holonomy condition, the curvature form at $(e,t)$ is a map
\[\Omega_t\colon\Lambda^2(\lie{g}\oplus\R)\to\su(3),\]
and by invariance the one-parameter family of forms $\Omega_t$ determines the curvature of $g$.

By the last remark of Section~\ref{sec:gauge}, $Q_{P_t}$ coincides with minus the Weingarten tensor; denoting by
\[\Omega_t^5\colon\Lambda^2\R^5\to\so(5)\]
the curvature forms of the five-dimensional metrics $g_t$, the Gauss equation gives
\begin{equation}
 \label{eqn:tangentialcurvature}
\Omega_t^{\text{tang}} = \Omega^5_t-\sum_{i,j} Q_{P_t}(e_i)\wedge Q_{P_t}(e_j) e_{i}\otimes e^j,
\end{equation}
where the tangential part $\Omega_t^{\text{tang}}$ of $g$ is related to $\Omega_t$ by the diagram
\[\xymatrix{\Lambda^2(\lie{g}\oplus\R)\ar[r]^{\Omega_t}&\so(6)\ar[d]^{i^T}\\ \Lambda^2\lie{g}\ar[u]^i\ar[r]^{\Omega_t^{\text{tang}}}&\so(5)}.\]
By the Bianchi identity $\Omega_t$ is symmetric. Hence
\[\im (\Omega_t\circ i)^\perp = \ker i^T\circ\Omega_t=\ker \Omega_t,\]
where we have used the fact that $\im\Omega_t\subset\su(3)$ and $i^T\colon\su(3)\to\so(5)$ is injective. It follows that
\[\dim\Omega_t^{\text{tang}}=\dim \im \Omega_t\circ i=\dim\im\Omega_t.\]
By the Ambrose-Singer theorem the metric has holonomy equal to $\SU(3)$ if and only if $\Omega_t^{\text{tang}}$ has rank $8$ for some $t$.

For example, in the case of $\mathcal{M}_3$, it follows easily from \eqref{eqn:tangentialcurvature} that the image of $\Omega_t^{\text{tang}}$  is spanned  by
\begin{gather*}
\frac{1}{2}  e^{34} {(\lambda^{2}-\mu^{2})}+ {(\lambda+\mu)^{2}} e^{12}, \quad  e^{24} {(\lambda^{2}-\mu^{2})}+  e^{13} {(\lambda^{2}-\mu^{2})},\\
  -e^{14} {(\lambda^{2}-\mu^{2})}+  e^{23} {(\lambda^{2}-\mu^{2})}, \quad  e^{15} {(3 \lambda^{2}+4  \mu \lambda+\mu^{2})}, \quad
 e^{25} {(3 \lambda^{2}+4  \mu \lambda+\mu^{2})}, \\
\frac{1}{2}  e^{12} {(\lambda^{2}-\mu^{2})}+ e^{34} {(\lambda- \mu)^{2}}, \quad
  e^{35} {(3 \lambda^{2}-4  \mu \lambda+\mu^{2})}, \quad
  e^{45} {(3 \lambda^{2}-4  \mu \lambda+\mu^{2})}.
\end{gather*}
Thus, the image is $8$-dimensional if and only if
\begin{equation}
 \label{diseqn:irreducible}
3 \lambda^{2}+\mu^{2}\neq\pm 4\mu\lambda, \quad \lambda\mu(\lambda^2-\mu^2)\neq0.
\end{equation}
It follows that the orbits on which $\lambda=\mu$ and $\mu=0$ do not give rise to irreducible six-dimensional metrics. On the other hand, if
\[(\lambda^2-\mu^2)^3=A\mu^4, \lambda,\mu>0, \quad A\neq0,\]
the inequalities \eqref{diseqn:irreducible} are generically satisfied. Hence the six-dimensional metric has holonomy equal to $\SU(3)$ in this case.

This can also be verified directly by integrating the metric as in the first example. For instance, if $A=-1$, setting 
\[s=\operatorname{arsinh}\lambda\mu^{-2/3},\]
whence
\[\mu=(\cosh s)^3, \quad \lambda=\sinh s\cosh^2 s, \quad s'=(\cosh s)^{-3},\]
one finds a metric on $G\times(0,+\infty)$, where $G$ is the Lie group with Lie algebra $(0,0,0,0,12+34)$ and an orthonormal frame is given by
\begin{multline*}
  \frac1{\cosh s \sqrt{1+\tanh s}} \eta^1, \frac1{\cosh s \sqrt{1+\tanh s}}  \eta^2 , 
  \sqrt{1+\tanh s}\,  \eta^3, 
-\sqrt{1+\tanh s} \,\eta^4, \\
\cosh s \,\eta^5, dt=\frac1{\cosh^{3} s} ds.
\end{multline*}

\end{example}
\section{Completeness}
In this section we show that our metrics cannot be extended in a complete way, except in the trivial case. In other words, we prove that every complete manifold with an integrable $\SU(3)$-structure preserved by the  cohomogeneity one action of a five-dimensional nilpotent Lie group is flat. The proof depends on the results of Section~\ref{sec:integrating} in an essential way.

Let $G$ be a $5$-dimensional Lie group acting  with cohomogeneity one on a complete manifold $M$ with an integrable $\SU(3)$-structure; assume that the action preserves the structure. Let $c(t)$ be a geodesic orthogonal to principal orbits. Then the map $(g,t)\to gc(t)$ defines a local diffeomorphism; the pullback to $G\times I$, with $I$ an interval, defines an integrable $\SU(3)$-structure, and therefore a one-parameter family of left-invariant hypo structures on $G$ satisfying the hypo evolution equations. In particular the metric defines an integral curve in $\mathcal{D}$ defined on the interval $I$; conversely, an integral curve in $\mathcal{D}$ defines a cohomogeneity one metric on some product $G\times I$.

If all orbits are five-dimensional, then one can assume $I=\R$;  this means that the one-parameter family of $\SU(2)$-structures is defined on all of $\R$. Otherwise there is an orbit of dimension less than five, called a \dfn{singular special orbit}. Since the action preseves a Riemannian metric, about a singular special orbit $M$ has the form 
\[G\times_H V,\]
where $H$ is the special stabilizer and $V$ its normal isotropy representation (see e.g. \cite{Bredon}). In this case, the interval $I$ has a boundary point $b\in\R$, with $c(b)$ lying in the singular special orbit;  we shall say that the metric defined by the integral curve in $\mathcal{D}$ \dfn{extends across $b$ with special stabilizer $H$}.

We can now state the main result of this section.
\begin{theorem}
\label{thm:main}
There are no singular special orbits on any six-manifold with an integrable $\SU(3)$-structure preserved by the cohomogeneity one action of a nilpotent Lie group of dimension five.
\end{theorem}
Notice that $M$ is not required to be complete in this theorem; if one adds this assumption, the following ensues:
\begin{corollary}
Let $M$ be a complete six-manifold with an integrable $\SU(3)$-structure preserved by the cohomogeneity one action of a nilpotent Lie group of dimension five. Then $M$ is flat.
\end{corollary}
\begin{proof}
By Theorem~\ref{thm:main}, all orbits are five-dimensional. Consequently, the map \[G\times\R\to M,\quad  (g,t)\to gc(t)\] is a local diffeomorphism, and the $\SU(3)$-structure can be pulled back. We can therefore  assume $M=G\times\R$. 

The geodesic $c(t)=(e,t)$ is a line on $M=G\times\R$; by the Cheeger-Gromoll theorem, the metric on $M$ is a product metric. This means that $\hat X$ preserves the metric: so it is both symmetric and antisymmetric, hence zero. So the evolution is trivial, and the defining forms $\alpha(t)$, $\omega_i(t)$ are constant in time and closed. This implies that $G$ is Ricci-flat, hence flat; so the product $G\times\R$ is also flat.
\end{proof}

In order to prove Theorem~\ref{thm:main}, we will need two lemmas.
\begin{lemma}
\label{lemma:stabilizer}
Suppose  $M$ has an $\SU(3)$-structure preserved by the cohomogeneity one action of a nilpotent five-dimensional Lie group $G$. Then the Lie algebra of each stabilizer is either trivial or an ideal of $\lie{g}$ isomorphic to $\R$.
\end{lemma}
\begin{proof}
Let $H$ be the stabilizer of $p\in M$; if $H$ is discrete, there is nothing to prove.

Otherwise, $Gp$ is a singular special orbit with a neighbourhood of the form $G\times_H V$, where 
\[T_p M =  T_p(Gp)\oplus V\]
is an orthogonal direct sum  and $H$ acts on $V$ via the isotropy representation. Then $H$ acts on $V$ with cohomogeneity one and discrete stabilizer.  Moreover $H$ acts transitively on a sphere, so it is a nilpotent Lie group that covers a sphere, and therefore necessarily isomorphic to $\R$ or $S^1$. Let $X$ be a generator of its Lie algebra $\lie{h}$. As representations of $H$, 
\[T_p (Gp) \cong \frac{\lie{g}}{\lie{h}},\]
which therefore has an invariant metric. Since the action of $H$ on the first factor is induced by the adjoint action, $\ad X$ is both nilpotent and antisymmetric, therefore zero. It follows that $\lie{h}$ is an  ideal of $\lie{g}$.
\end{proof}

\begin{lemma}
\label{lemma:intexp}
Let $\sigma\colon I\to\mathcal{D}$ be a maximal integral curve for the hypo evolution flow, and assume the induced cohomogeneity one metric extends across the boundary point $b$ of $I$ with special stabilizer $H$. Suppose that $W\subseteq\R^5$ is a linear subspace invariant under $\hat X_{\sigma(t)}$ for all $t$. Then either
\begin{align*}
\int_{t_0}^{b} \operatorname{tr}((\hat X_{\sigma(t)})|_{W}&=-\infty, \quad\lie{h}\subset W; \quad \text{or}\\
\abs{\int_{t_0}^{b} \operatorname{tr}((\hat X_{\sigma(t)})|_{W}}&<+\infty,\quad \lie{h}\not\subset W.
\end{align*}
\end{lemma}
\begin{proof}
By construction the metric on $G\times_H V$ pulls back to a time-dependent symmetric tensor $g$ on $G$, which degenerates at $t=b$, and satisfies
\[g'(t)=\hat X_{\sigma(t)}g(t).\]
Because $\hat X_{\sigma(t)}$ is symmetric it also leaves $W^\perp$ invariant; this implies that $W$ stays orthogonal to $W^\perp$ for all $t$, and
\[(g'(t)|_W)=\hat X_{\sigma(t)}|_Wg|_W(t).\]
The determinant of $g(t)|_W$ satisfies
\[(\det g(t)|_W) ' = \operatorname{tr}(\hat X_{\sigma(t)}|_W) \det g|_W(t),\]
so for any fixed $t_0$ in $I$, 
\[\det g|_W(b)=\left(\exp \int_{t_0}^b \operatorname{tr}(\hat X_{\sigma(t)}|_W)dt\right)\det g|_W(t_0).\]
It suffices now to observe that $\det g|_W(b)$ may not be infinite, and it is zero if and only if $W$ contains $\lie{h}$.
\end{proof}
This lemma will be mostly used in the case that $W$ is one-dimensional, i.e. that $\hat X$ has an eigenvector not depending on $t$, and $W=\R^5$, in which case necessarily $\lie{h}\subset W$.

\begin{proof}[Proof of Theorem~\ref{thm:main}]
Let $M$ be as in the statement, and suppose $M$ has a singular special orbit; then there is a maximal integral curve $\sigma\colon I\to \mathcal{D}$ whose associated metric extends across a boundary point $b$ of $I$ with special stabilizer $H$. Because the symmetry group is nilpotent, by Theorem~\ref{thm:hierarchy}  we can assume up to $U(2)$ action that the integral curve is contained in one of the families $\mathcal{M}_1$, $\mathcal{M}_2$ and $\mathcal{M}_3$. We will discuss each case separately.

For $\mathcal{M}_1$, observe that by the proof of Lemma~\ref{lemma:family1} 
\[\operatorname{tr}\hat X_{\sigma(t)} = \frac12(\mu(t)+k(t)).\]
So by Lemma~\ref{lemma:intexp} applied to $W=\R^5$, necessarily
\[\int_{t_0}^b (\mu+k)dt=-\infty.\]
Applying the same lemma to $W=\Span{e_5}$, we see that $\lie{h}$ is spanned by $e_5$. By Lemma~\ref{lemma:stabilizer}, it follows that $e_5$ is in the center of $\lie{g}$, and by definition of $\mathcal{M}_1$ this is equivalent to 
\[\lambda=h=k=\mu=0.\]
However, this condition implies that the integral curve is constant and defined on all of $\R$, which is absurd.

\smallskip
In the case of $\mathcal{M}_2$, the explicit formulae appearing in the proof of Theorem~\ref{thm:evolveM2} give
\[\operatorname{tr}\hat X_{\sigma(t)} = -\frac12(\mu+k).\]
So, applying  Lemma~\ref{lemma:intexp} to $W=\R^5$,
\[\int_{t_0}^b -(\mu+k)dt=-\infty.\]
Applying Lemma~\ref{lemma:intexp} to $W=\Span{e_1,e_3}$, we see that $\lie{h}$ is contained in $W$. By Lemma~\ref{lemma:stabilizer}, $\lie{h}$ is a one-dimensional ideal. This gives rise to two possibilities:

\emph{i}) If $x=\lambda=y=\mu=0$, we can apply Lemma~\ref{lemma:intexp} to $\Span{e_5}$, which is a constant eigenspace with eigenvalue $k/2$, and reach a contradiction.

\emph{ii}) If
\[x=h=\lambda=0,\]
the integral curve is some $O_l$ in $\mathcal{M}_2$. Even without using the explicit expression of the metric appearing in Section~\ref{sec:examples}, one can compute 
\[\mu(t)=\frac{\mu(0)}{1-\frac32t}, \quad k(t)=\frac{k(0)}{1-\frac32t}.\]
Therefore the interval of definition is $(-\infty,\frac23)$. Now we observe that $\hat X$ has constant eigenvectors with eigenvalues $0$ and $\pm\frac12(k+\mu)$. Therefore
\[\int_{t_0}^{2/3} \mu+k =  +\infty\]
contradicts Lemma~\ref{lemma:intexp}. 

\smallskip

Finally, for $\mathcal{M}_3$ the trace of $\hat X_{\sigma(t)}$
is $-\lambda$, so Lemma~\ref{lemma:intexp} applied to $W=\R^5$ gives
\[\int_{t_0}^b- \lambda=-\infty.\]
On the other hand $e_5$ is a costant eigenvector with eigenvalue $\lambda$, so applying Lemma~\ref{lemma:intexp} again gives a contradiction.
\end{proof}

\bibliographystyle{plain}
\bibliography{violent}

\small\noindent Dipartimento di Matematica e Applicazioni, Universit\`a di Milano Bicocca, via Cozzi 53, 20125 Milano, Italy.\\
\texttt{diego.conti@unimib.it}

\end{document}